\documentclass[12pt]{amsart}
\usepackage{amssymb}
\usepackage{amsthm}
\usepackage{amsmath}
\usepackage{amsfonts}
\usepackage{stackrel}
\usepackage{appendix}
\usepackage{enumitem}
\setlist{topsep=.3cm, itemsep=.2cm}
\usepackage{mathrsfs}
\usepackage{color}
\usepackage[dvipsnames]{xcolor}
\usepackage{appendix}

\usepackage[T1]{fontenc}
\newtheorem{thm}{Theorem}[section]
\newtheorem{defi}[thm]{Definition}
\newtheorem{pro}[thm]{Proposition}
\newtheorem{cor}[thm]{Corollary}
\newtheorem{lem}[thm]{Lemma}
\newtheorem{rem}[thm]{Remark}

\numberwithin{equation}{section}
\usepackage{hyperref}
\hypersetup{colorlinks=true,linkcolor=blue,filecolor=blue, citecolor=blue,urlcolor=cyan}
\usepackage{comment}

\newcommand{\R}{\mathbb{R}}
\newcommand{\N}{\mathbb{N}}

\newcommand{\D}{\displaystyle}

\def\qq#1{\qquad \mbox{#1}\quad}
\def\qt#1{\quad \mbox{#1}\ }
\def\q#1{\quad \mbox{#1}\ }
\newcommand{\al}{\alpha}
\newcommand{\be}{\beta}
\newcommand{\De}{\Delta}
\newcommand{\de}{\delta}

\newcommand{\e}{\varepsilon}
\newcommand{\eps}{\e}

\newcommand{\Ga}{\Gamma}
\newcommand{\la}{\lambda}
\newcommand{\La}{\Lambda}

\newcommand{\na}{\nabla}
\newcommand{\Om}{\Omega}
\newcommand{\om}{\omega}
\newcommand{\Omb}{\overline{\Om}}
\newcommand{\p}{\partial}
\newcommand{\s}{\sigma}

\newcommand{\vf}{\varphi}

\newcommand{\wto}{\rightharpoonup}

\title[A $p$-Laplacian Problem with Slightly Subcritical  Nonlinearity]
{A $p$-Laplacian Problem with Slightly Subcritical Regularly Varying Nonlinearity}

\author{Mabel Cuesta}
\address[M.~Cuesta]{Department of Mathematics, Universit\'e du Littoral C\^ote d'Opale (ULCO), Laboratoire de Math\'ematiques Pures et Appliqu\'ees Joseph Liouville (LMPA),
	62100 Calais, France.
Orcid ID : 0000-0002-8576-3941}
\author{Rosa Pardo}
\address[R.~Pardo]{Departamento de An\'alisis Matem\'atico y Matem\'atica Aplicada,  Universidad Complutense de Madrid, 28040--Madrid, Spain.
Orcid ID : 0000-0003-1914-9203}
\thanks{The second author is supported by grants  PID2022-137074NB-I00,  MICINN,  Spain, and by UCM, Spain,  Grupo 920894.}

\date{}

\begin{document}
\maketitle

{\it Positive solutions,  subcritical nonlinearity, changing sign weight,  $p$-Laplacian, regularly varying functions}

\begin{abstract}
We study a quasilinear elliptic problem involving the $p$-laplacian operator  and a  slightly subcritical nonlinearity with a sign-changing weight. 
We assume  that the slightly subcritical nonlinearity   is a  regularly varying function at zero and at infinity,   which are not necessarily
asymptotic to  a power at infinity.
We state sufficient conditions which guarantee a Palais-Smale condition.
We also provide  a bifurcation theorem for these nonlinearities,  which allow us to state  the existence of a bifurcated branch of  positive solutions, containing a turning point, and   multiplicity of solutions. 
\end{abstract}
\maketitle
\section{Introduction}
We are interested in the existence of positive solutions to a Dirichlet problem for a class of quasilinear elliptic equations whose nonlinear term is  slightly subcritical and  regularly varying at infinity, see \ref{f1}  and \ref{f3}, and  involves an indefinite weight. 
More precisely,  we consider the quasilinear problem 
\begin{equation}
\label{eq:ell:pb} 
-\De_p u
=\la u^{p-1}+m(x)f(u)  \quad \mbox{in } \Om, \qquad u= 0,  \quad \mbox{on } \partial \Om,
\end{equation}
where   $ 1<p<N$,  $\Om \subset \R ^N $    is a bounded, connected open subset, with $C^{1,\alpha}$ boundary $\partial \Om$, $\De_p(u)=\text{div }(|\na u|^{p-2}\na u)$ is the $p$-Laplacian operator, $\la\in\R$ is a real parameter,   $m\in C(\overline{\Om})$  changes sign in $\Om$, 
and $f \in C([0,+\infty))$, $f(0)=0$. Solutions of \eqref{eq:ell:pb} will be understood in the weak sense.

\medskip

When $f(s)$ is $p$-superlinear and  grows at infinity as  $s^{q-1}$ with $p-1<q<p* :=\frac{Np}{N-p}$ the   critical Sobolev exponent, this problem is widely studied. In particular, the case $p=2$ and $a$ changing sign was first studied by \cite{Alama-Tarantello} for a power-like nonlinearity  via local bifurcation, monotonicity  and variational methods. Next by  \cite{BCN} via degree theory, see also   \cite{Cuesta_Pardo_MJM_2022}, and references therein. For the case  $p\not=2$, see \cite{Bo-Ta, Ilyasov, Kaufman_Ramos-Quoirin_Umezu} for power type nonlinearities $f$ among  others. 
\medskip

We are interested in the solvability of problem \eqref{eq:ell:pb} in the case when  $f$ does not grow as a subcritical power. For instance,   when $p=2$ and $f(s):=O\left(\dfrac{|s|^{2^* -2}s}{[\ln(e+|s|)]^\alpha}\right)$, this problem has been studied in \cite{Cuesta_Pardo_MJM_2022}, see also  \cite{Castro_Pardo_RMC_2015, Castro_Pardo_DCDS_2017, Clapp_Pardo_Pistoia_Saldana} for this nonlinearity, and \cite{Cuesta_Pardo_Pistoia_MJM_2024} for a related system. Our aim is to explore this results for $p\not=2$  for  nonlinearities $f$   $p$-superlinear and slightly subcritical.

\medskip

Prototype nonlinearities are the following ones
\begin{equation}\label{protot:f}
f(s):=\dfrac{|s|^{p^* -2}s}{\big[\ln(e+|s|)\big]^\alpha},\
f(s):=\dfrac{|s|^{p^* -2}s}{\Big[\ln\big(e+\ln(1+|s|)\big)\Big]^\alpha}, \ 
s>0, 
\end{equation}
where $\al>0$ is a fixed exponent, or even having a smaller perturbation  at infinity.  

\medskip
Instead of focusing on the prototypes, we shall consider a wider class of nonlinearities.
Precisely,  we will assume   the following hypothesis :

\smallskip

\begin{enumerate}[label=\textbf{(f\arabic*)$_\infty$}]
\item
\label{f1}
$f$ is slightly subcritical at infinity : ${\D\lim_{s\to +\infty}}\, \frac{f(s)}{|s|^{p^*-1}}=0$;

\item 
\label{f2}
$f>0$  at infinity and for all $ K\subset \R $ compact, there exists a constant  $c_1\geq 0$, 
\begin{equation}
\label{c1}
 f(s)+c_1s^{p-1}  \hbox{ is increasing for } s\in K,\ s>0;
\end{equation}

\item
\label{f3}
there exist  constants $c_2>0$ and $ q\in (p-1,p^*-1]$ 
such that for all $\tau>0$, 
\begin{equation}\label{bf}
\frac{f(\tau s)}{f(s)}\le c_2(1+\tau^{p^*-1})\ \text{ for all } s\text{ big enough,}  
\end{equation}
and
\begin{equation}\label{reg:var}
{\D\lim_{s\to +\infty}}\, \frac{f(\tau s)}{f(s)}= \tau^q, \qquad \forall \tau >0.
\end{equation}
\end{enumerate}

Positive measurable functions $f$ satisfying \eqref{reg:var} for some $q\in\R$ are known as {\it regularly varying functions  (at infinity)}, and are denoted as   $f\in RV_q,$ see  the seminal work of Karamata \cite{Karamata}, see also the books \cite{Bingham, Resnick, Seneta}. Those functions do not need to be asymptotic to a power, see for instance \eqref{protot:f} or \eqref{crit:exp}.\smallskip

We will use a variational approach needing the Palais-Smale condition of the associated  energy functional. This is a delicate issue for slightly subcritical nonlinearities, where in general Rellich–Kondrachov Theorem on compact Sobolev embeddings into Lebesgue spaces do not apply (see  \cite[Theorem 9.16]{Brezis_2011}).\smallskip

To guarantee some kind of compactness, we resort  a careful analysis of the behaviour of $f$ at infinity that relies on Lebesgue-Vitali's convergence Theorem (see \cite[Proposition  \ref{pro:PS} and Theorem \ref{th:comp:PS}]{Bogachev}).

Let us stress  that, if moreover, $f$ is increasing at infinity, techniques  coming from Orlicz spaces can be used, see \cite[Proposition 3.1 and Theorem A.2]{Cuesta_Pardo_MN_2024}.  Here, instead,  we only ask a mild monotonicity on compact sets, c.f. \ref{f2}, a hypothesis used  mainly to  apply some comparison results, see Proposition \ref{pro:sol:min}  and Appendix A.

\medskip

In order to prove the existence of positive solutions, we will use local and global bifurcation and  variational techniques.   Using Orlicz spaces, we extended in \cite[Theorem 3.2]{Cuesta_Pardo_MN_2024}  the well  know bifurcation results of  Drabek for the $p$-Laplacian problems,  and provide a bifurcation theorem  generalizing it to the case of slightly subcritical nonlinearities, and allowing a changing sign weight in the $p$-linear term.   In  Theorem \ref{th:ex2} 
we state  a similar global bifurcation result  and we give some hits to adapt the proof of \cite[Theorem 3.2]{Cuesta_Pardo_MN_2024}  using Lebesgue-Vitali convergence theorem. We refer to Appendix C for further details.

\medskip

Bifurcation from zero methods requires  some hypothesis on the behaviour of $f$ close to 0. We will assume  in this work  the following conditions:

\smallskip

\newcounter{num}
\setcounter{num}{4}
\begin{enumerate}[label=\textbf{(f\thenum)$_0$}\addtocounter{num}{1}]
\item
\label{H5}
$f(0)=0,$ there exists $\de_0>0$ such that $f(s)>0$ for $s\in(0,\de_0)$, and $f$ is $p$-sublinear at zero : $\ \displaystyle{\lim_{s\to 0^+}\frac{f(s)}{s^{p-1}}=0}$;

\item

\label{H6}
there exists a constant   $\eta>0$ such that 
\begin{equation}\label{r}
\lim_{s\to 0^+}\, \dfrac{f(\tau s)}{f(s)}=
\tau^\eta\,\quad  \qq{for all} \tau\geq 0
\end{equation}
and  that 
\begin{equation}
\label{a:vf1:<0}
I_0:=\int_{\Om} m(x)\vf_1^{\eta+1}\, dx <0
\end{equation}
where $\varphi_1$ is the unique  positive eigenfunction for $-\Delta_p$ in $W_0^{1,p}(\Om)$ normalised in the $L^\infty(\Om)$-norm.
\end{enumerate}

\bigskip 

As before, positive measurable functions $f$ satisfying \eqref{r} for some $\eta\in\R$ are known as {\it regularly varying functions  (at 0)}.  
\bigskip

Concerning the weight $m$, we assume the following hypothesis. All throughout the paper, for $v:\Om\to\R$, $v=v^+-v^-$ where
\[
v^+(x):=\max\{v(x),0\} \qq{and}  v^-(x):=\max\{-v(x),0\}.
\]
Let us  define

\begin{equation*}\begin{array}{ll}
\Om^\pm =:\{x\in \Om : \ \pm m(x)>0\},\quad      
&  \Om^0 =:\{x\in \Om : \ m(x)  = 0\},\\[.3cm]
 \om^{+,0} =:{\rm int}\,\big(\Om^+\cup \Om^0\big),\quad     
 & \om^{0}=:{\rm int}\,\big(\Om^0\big),
\end{array}
\end{equation*} 
and assume that  
\begin{enumerate}[label=\textbf{(m\arabic*)}]
\item 
\label{m0}
$\sup m^+>0$ and $\sup m^->0$, so both $\Om^\pm \ne\emptyset $;
\item
\label{m2}  the {\it thick zero set condition} is satisfied: $\overline{\Om^+}\, \cap\, \overline{\Om^-}=\emptyset$, moreover 
we assume that  $\partial\om^0$ is of class $C^1$. 
\end{enumerate}

Observe that \ref{m2} implies that $\om^0\not=\emptyset $. The thick zero set condition was first stated  by Alama and Tarantello in \cite{Alama-Tarantello}. It is an usual condition on superlinear problems with changing sign weights,  although  other conditions to replace the thick zero set condition have been explored for instance by  \cite{Ramos_Terracini_Troestler} and more recently by 
for a slightly superlinear  problem.

\bigskip

Let us denote by $\la_1(\om)$ the least  eigenvalue of the Dirichlet eigenvalue problem 
\begin{equation*}
-\De_p \vf=\la\vf^{p-1} \qt{in}\om,\qquad \vf= 0\qt{on}\p \om, 
\end{equation*}
where $\om\subset\R^N$ is an open bounded set.
We denote by $\vf_1(\om)$   the eigenfunction associated  to $\la_1(\om)$  with $\|\varphi_1(\omega)\|_\infty=1$ and $\vf_1(\om)\geq 0$, $\vf_1 (\om)\not=0$.
It is known \cite{Anane} that if    $\om$ is connected,  the eigenvalue  $\la_1 (\om)$ is simple, isolated  and it is the unique eigenvalue having a  definite-sign eigenfunction; furthermore $\vf_1(\om)>0$ in $\om$.
We  write $(\la_1,\vf_1)$ when $\om=\Om$. 

\medskip
Denote
\begin{equation}\label{C0}
C_0=\inf\{C\geq 0:\, f(s)+Cs^{p-1} > 0 \hbox{ for all } s> 0\},
\end{equation}
and remark that  \ref{f2} and \ref{H5}  implies that  $C_0<+\infty$.

\medskip 

Our main result is the following theorem. 

\begin{thm}\label{th:ex}
Let $1<p<N$ and  assume that $f$  satisfies hypothesis {\rm \ref{f1}}--{\rm \ref{f3}} and {\rm \ref{H5}}--{\rm \ref{H6}}.	
If $m\in C(\Omb)$  satisfies hypothesis {\rm \ref{m0}}--{\rm \ref{m2}}, 
and the sign condition \eqref{a:vf1:<0} holds, then there exists  $\La\in\R$ satisfying 
\begin{equation}
\label{La}
\la_1<\La\le\min\Big\{\la_1\big(\om^{0}\big)\, ,\
\la_1\big(\om^{+,0}\big) +C_0\sup m^+\Big\},
\end{equation}
and such that   \eqref{eq:ell:pb} has a  positive solution for any  $\la < \La$, and there is not non-negative nontrivial solution to \eqref{eq:ell:pb} if  $\la > \La$. 
Besides, there exists an unbounded,  closed and connected set $ \mathscr{C}^+$ of  positive  solutions to \eqref{eq:ell:pb} bifurcating from the trivial solution set at the bifurcation point $(\la,u)=(\la_1,0)$.

Moreover,
\begin{itemize}
\item[\rm (a)] For every, 
$\la\in\big(\la_1, \La)$,  \eqref{eq:ell:pb} admits at least two  ordered positive solutions, one of those is a local minimum of the energy functional.

\item[\rm (b)] If $\La< \la_1\big(\om^{0}\big)$ then, problem  \eqref{eq:ell:pb} admits at least one   positive solution  for  $\la=\La$.

\item[\rm (c)] For every $\la\le\la_1$, problem  \eqref{eq:ell:pb} admits at least  one   positive solution.
\end{itemize}
\end{thm}

\begin{rem}\label{reguC1}\rm
	(i) We recall that  $(\la,u)\in \R\times W^{1,p}_0(\Om)$, with $u\geq 0 $,
	is a   \textit{ weak solution} to \eqref{eq:ell:pb} if  for each  $\psi\in W_0^{1,p}(\Om)$,
	\[
	\int_{\Om}|\na u|^{p-2}\na u\cdot \na \psi\,\, dx=  \int_{\Om}\Big(\la u^{\, p-1}+m(x)f(u)\Big) \psi \, dx .
	\]
	
	(ii) By  the  standard regularity  results for quasilinear problem, see \cite{DiBenedetto, Lieberman_bdry_reg},   there exists $\mu\in(0,1)$ such that weak solutions of \eqref{eq:ell:pb} belong  to $ C^{1,\mu} (\overline{\Om})$. 
	From the strong maximum principles, see \cite{Pucci_Serrin_2007, Pucci_Serrin_Zou, Vazquez_MP}, any positive solution $u$ satisfies $u>0$ in $\Omega$. In fact $u\in {\mathscr P}^+$, where
	\begin{equation}\label{def:P+}
	{\mathscr P}^+:=\left\{u\in C_0^1(\overline{\Om}): u>0 \hbox{ in } \Om \hbox{ and } \frac{\partial u}{\partial \nu}<0 \hbox{ on } \partial \Om \right\}.    
	\end{equation}

	(iii) Observe that whenever $f>0$ for $s>0$, $C_0=0$,  so $\La< \la_1\big(\om^{0}\big)$  when $\om^{+,0}$ is connected, since $\la_1\big(\om^{+,0}\big)<\la_1\big(\om^{0}\big)$.\\
\end{rem}

The main novelties of this paper are the following:
\begin{enumerate}
\item[(i)] 
In Theorem \ref{th:ex} we generalize  results of Alama and Tarantello \cite{Alama-Tarantello} on multiplicity of solutions,
for $p \ne 2$ and regularly varying functions. To be able to do that, we  adapt variational methods and bifurcation techniques  for $p$-laplacian problems with slightly subcritical nonlinearities  of non power-like nonlinearities.

\item[(ii)]  Rellich–Kondrachov Theorem \cite[Theorem 9.16]{Brezis_2011} on compact embedding of Sobolev spaces into Lebesgue spaces,  can not be applied   in presence of slightly subcritical nonlinearities.  Instead, compactness is achieved,  via   the Lebesgue-Vitali's convergence Theorem,
see  Proposition \ref{pro:compG}. 

\item[(iii)]  We develop variational methods. 
In Proposition \ref{pro:PS} and Theorem \ref{th:comp:PS}, we provide  sufficient conditions guaranteeing that  \textit{ Palais-Smale  sequences} are bounded and posses a strongly convergent subsequence. 
Thus, we obtain positive solutions  by  means of a Mountain--Pass Theorem, see Theorem \ref{th:la<la1}. 

\item[(iv)]  In Theorem \ref{pro:bif drabek}, we broaden previously known bifurcation results  for a $p$-laplacian problem with slightly subcritical nonlinearities, see \cite{Cuesta_Pardo_MN_2024}. In Theorem \ref{th:ex2}, the monotonicity conditions that ensure  the positivity of the solutions are weakened,   complementing the results in \cite{Drabek, DelPino-Manasevich, Girg_Takac, Cuesta_Pardo_MN_2024}.
\end{enumerate}

\medskip

This paper is organised in the following way.  Section \ref{section: rvf} summarise some results on regularly varying functions that will be used later. Section 3 adapt a bifurcation result and   a non-existence result, coming from \cite{Cuesta_Pardo_MN_2024}.
Section \ref{sec:exist} deals with the (PS)-condition, and the existence of solutions  for $\la\le\la_1$, based in the Mountain--Pass Theorem. 
The proof of Theorem \ref{th:ex} is completed in Section \ref{sec:proof}.

\medskip 

Hereafter we denote by $\|\cdot\|$ the usual norm in $W_0^{1,p} (\Om)$:
$$
\|u\|=\left(\int_\Om |\na u|^p\, dx\right)^{1/p}.
$$

\section{Some properties of regularly varying functions}\label{section: rvf}
Positive measurable functions $f$ satisfying \eqref{reg:var} for all $\tau>0$ are known as 
{\it regularly varying functions (at infinity) of index $q\in\R$}, see  the classical reference \cite{Karamata}, and \cite{Bingham, Resnick, Seneta} for the general theory. For  applications to the laplacian operator ($p=2$), see \cite{G-Melian_Iturriaga_RQuoirin}  with nonlinearities such as $f(s)=s\big(\log(1+s)\big)^\al$ with  $\al>0$, and \cite{Costa_Quoirin_Tehrani} for  semipositone elliptic problems with nonlinearities such as $f(s)=s^q\big(\log(1+s)\big)^\al$ for $q\in(1,2^*-1)$ and $\al>0$. \\

Since Seneta prove that the limit \eqref{reg:var}  is uniform for  $\tau$  on compact intervals of $(0,+\infty)$, see the uniform convergence theorem  \cite[Theorem 1.1]{Seneta}, then \eqref{reg:var} implies \eqref{bf} for  $\tau$  on compact intervals of $(0,+\infty)$. Indeed, let $\tau\in[a,b]\subset (0,+\infty)$, the  above uniform convergence theorem implies that for all $\tau\in[a,b]$ there exists $s_0>0$ such that
\begin{equation*}
\frac{f(\tau s)}{f(s)}\le 1+ \tau^q, \quad \forall s >s_0,\quad \forall \tau\in[a,b].
\end{equation*}
Condition \eqref{bf} ask for uniformity  of the constant, for all $\tau>0.$\\

Obviously functions asymptotic to the power $s^q$ at infinity satisfy 
\eqref{reg:var}.
But, there are  functions  satisfying \eqref{reg:var} which are not 
asymptotic to  $s^q$ at infinity, see for instance $f(s)=s^q(\ln(e+s))^\al$, where $q\in(1,2^*-1)$ and $\al>0 ,$ or the prototype functions given by \eqref{protot:f}, or  
\begin{equation}\label{crit:exp}
f(s)=\frac{|s|^{2^* -2}s}{\exp{\Big(\big[\ln(e+|s|)\big]^\be\Big)}} \qq{for} 0<\be<1. 
\end{equation}

Observe that for any $\e>0$ fixed,
$$
\lim_{s\to+\infty}\frac{\exp{\Big(\big[\ln(e+|s|)\big]^\be\Big)}}{s^{\e}}=\lim_{t\to+\infty}\frac{e^{(t^\be)}}{e^{\e t}}\left(\frac{e^{t}}{e^t-e}\right)^\e=0.
$$

Although the following results can be stated for purely positive measurable and locally integrable functions, we will state them for continuous functions.

First,  the Karamata Representation Theorem holds for regularly
varying functions, see \cite[Theorem III]{Karamata}, cf. \cite[Corollary to Theorem 0.6]{Resnick}, cf. also  \cite[Theorem 1.2 \& Lemma 2.1]{Seneta}.

\begin{thm}[Karamata representation Theorem.] \label{rrvf}
Let $f\in C(0,+\infty)$ be positive at infinity and regularly varying at infinity of order $q\in\R$. Then, there exists a regularly varying function (at infinity) of index 0 {\it (slowly varying)} denoted by $L$ such that
\begin{equation}\label{f:q:L}
f(s)=s^q L(s),
\end{equation}
where $L$ has the canonical form
\begin{equation}\label{L}
L(s)=\zeta(s){\rm exp}\left(\int_{s_0}^s \frac{\rho(t)}{t} dt\right), \quad \forall s\geq s_0,    
\end{equation}
for some $s_0>0$, and where $\zeta$ is a  positive, continuous function  with a positive limit at infinity, and $\rho$ is  a continuous function  such that $\displaystyle\lim_{s\to+\infty}\rho(s)=0.$
\end{thm}

As a corollary, the following property   is proved in \cite[Section 1.5, pg.18]{Seneta}:
\begin{cor}\label{cor:superlinear}
It follows from Theorem \ref{rrvf} that for all $q'<q$,
\begin{equation}\label{superp}
 \lim_{s\to+\infty}\frac{f(s)}{s^{q'}}=+\infty. 
\end{equation}
In particular \eqref{superp} holds for $q'=p-1$ when we assume  $q\in(p-1,p^*-1]$.
\end{cor}

\medskip 

The  following Uniform Convergence Theorem for measurable slowly varying functions was proved by Karamata in the continuous case, 
see also \cite[Theorem 1.1]{Seneta} and \cite[Theorem 1.5.2]{Bingham}.

\begin{thm}[Uniform Convergence Theorem]
    \label{usfF}
Let $f\in C(0,+\infty)$ positive at infinity and satisfying  \eqref{reg:var} for some $q\in\R$.  Let $F(s)=:\int_0^s f(t)dt$. 
 Then 
\begin{enumerate}[label=\rm{(\roman*)}]
\item  $F$ is a function regularly varying at infinity of order $q+1$:
 \begin{equation}
 \label{reg:var:F}   
 \lim_{s\to+\infty}\frac{F(\tau s)}{F(s)}=\tau^{q+1},\qquad \forall \tau>0,
 \end{equation}
 and satisfies
\begin{align} 
\lim_{s\to+\infty}\frac{sf(s)}{F(s)}=q+1.
\label{eq:usfF}
\end{align}
\item  Furthermore, for all $a,  b\in \R^+,$ $0<a<b$,  the limits  \eqref{reg:var} and \eqref{reg:var:F} are  uniform for $ \tau \in [a,b].$
\item  Moreover, if $f$ is bounded on the interval $[0,s_0]$, for $s_0$ defined by \eqref{L}, then
the limits  \eqref{reg:var} and \eqref{reg:var:F} are  uniform for $ \tau \in (0,b]$ if $q>0.$ 
\end{enumerate} 
\end{thm}

 \bigskip

\section{Bifurcation results and a non-existence result}

In this section we include new results on bifurcation for slightly subcritical nonlinearities that have been recently studied    in \cite{Cuesta_Pardo_MN_2024}.   See \cite[Theorem 3.2]{Cuesta_Pardo_MN_2024} for a global bifurcation result with Rabinowitz's alternative, and \cite[Theorem 3.4]{Cuesta_Pardo_MN_2024} for sufficient conditions guaranteeing the positivity of the solutions in the bifurcated branch.

\medskip 

As an application, let us  state the  following theorem that  provide  the existence of a continuum of positive solutions to \eqref{eq:ell:pb},
bifurcating  to the right at $(\la_1,0)$. \\

\begin{thm}\label{th:ex2}
Assume that $f$  satisfies hypothesis {\rm {\rm \ref{f1}}}{\textbf -}{\rm \ref{f2}} and also
{\rm \ref{H5}}{\textbf -}{\rm \ref{H6}}.
Assume further that $m\in C(\Omb)$ 
satisfies hypothesis {\rm \ref{m0}}.
Then 
\begin{enumerate}
\item[\rm (i)] there exists an  unbounded, closed, connected set $ \mathscr{C}^+\subset \R\times W_0^{1,p}(\Om)$  of  positive  solutions to \eqref{eq:ell:pb},  bifurcating from the trivial solution set at the bifurcation point $(\la_1,0)$; 

\item[\rm (ii)]
$ \mathscr{C}^+$ bifurcates to the right. Precisely,
there exists $\delta>0$ such that 
\begin{equation*}
(\la,u)\in {\mathscr C}^+, 0<|\la-\la_1|<\delta, \,  0\not=\|u\|<\delta \Longrightarrow \la>\la_1;
\end{equation*}

\item[\rm (iii)]  if we assume  {\rm \ref{m2}},  there exist   $ \La>\la_1,$ such that  if problem \eqref{eq:ell:pb} has  non-negative non-trivial solutions, then    $\la \le \La$. 
\end{enumerate}
\end{thm}
\begin{proof}
We only have to adapt the proofs of \cite{Cuesta_Pardo_MN_2024} to our case ($V(x)\equiv 1$), firstly \cite[Theorem 3.4]{Cuesta_Pardo_MN_2024}, and secondly \cite[Theorem 1.1]{Cuesta_Pardo_MN_2024}. Notice  that  the operator $G$ is compact, according to Proposition \ref{pro:compG}. 
\end{proof}

We end this section with a result on an upper bound of  the projection of 
$ {\mathscr C}^+$  on the parameter space, $\La$ (see  \cite[Proposition 4.1]{Cuesta_Pardo_MN_2024}).

\begin{pro}[A non-existence result]
\label{pro:non:exist}
Let  $f$ satisfy hypothesis   {\rm \ref{f1}} and {\rm \ref{H5}}. Assume that $m\in C(\Omb)$  satisfies {\rm \ref{m0}}.
\begin{enumerate}
\item[(i)] Problem   \eqref{eq:ell:pb} does not admit a non-negative nontrivial solution $u\in  W_0^{1,p} (\Om)$ for any 
$$
\la\ge\la_1\big(\om^{+,0}\big)+C_0\sup m^+,
$$
where $C_0$ is  defined  in \eqref{C0}.
\item[(ii)] Assume also hypothesis {\rm \ref{m2}}.
Problem   \eqref{eq:ell:pb} does not admit a non-negative nontrivial solution for any 
$$
\la\ge\la_1\big(\om^{0}\big).
$$
\end{enumerate}
\end{pro}

\section{On Palais-Smale sequences, and on an existence result for \texorpdfstring{$\la\le\lambda_1$}{}}
\label{sec:exist}
In this section we prove  that  \textit{ Palais-Smale  sequences} are bounded and possesses a  strongly convergent subsequence, whenever $\la<\la_1\big(\om^{0}\big)$,  see Proposition \ref{pro:PS} and Theorem \ref{th:comp:PS} respectively. 

It is  standard to distinguish the cases $\lambda<\lambda_1$,    $\lambda=\lambda_1$ and $\lambda>\lambda_1$. In the first two cases, positive solutions can be obtained by a Mountain--Pass Theorem, see Theorem \ref{th:la<la1}.

\subsection{On Palais-Smale sequences}

Let us consider the functional $J_\la : W_0^{1,p}(\Om)\to\R$ associated to the problem  \eqref{eq:ell:pb}$_\la$, given by
\begin{align}
\label{def:J:lambda}
J_\la[u]&:=
\frac{1}{p}\int_{\Om}  |\na u|^p\, \, dx -
\frac{\la}{p}\int_{\Om}(u^+)^p\, dx -\int_{\Om}m(x)F(u^+)\, dx. 
\end{align}

The functional $J_\la$ is well defined  and of class $C^1$, with
\begin{align}
\langle J_\la'[u], \psi \rangle
&\,=\int_{\Om} |\na u|^{p-2} \na u\cdot\na \psi\,\, dx  -
\la\int_{\Om}(u^+)^{p-1}\psi \, dx -\int_{\Om}m(x)f(u^+)\psi\, dx ,\nonumber 
\end{align}
for all $\psi\in W_0^{1,p}(\Om).$

\begin{rem}\label{sol:+}
If $f(0)=0$, the choice $\psi=u^-$ gives $\langle J_\la'[u], u^- \rangle =-\int_{\Om} |\na u^-|^{p}\, dx $.
Consequently, nontrivial critical points of the functional $J_\la$ correspond to 
non-negative weak solutions to  \eqref{eq:ell:pb}$_\la$. By Remark \ref{reguC1}.(ii),  nontrivial critical points of the functional $J_\la$ are positive solutions to the problem  \eqref{eq:ell:pb}$_\la$.     
\end{rem}

\medskip
The next Proposition proves in particular that  \textit{ Palais-Smale  sequences} are bounded 
whenever $\la$ is smaller than $\la_1\big(\om^{0}\big)$.

\begin{pro}\label{pro:PS}
Assume  that   hypotheses  {\rm \ref{f1}}--{\rm \ref{f3}} are satisfied. Assume also that  
$m \in C(\Omb)$,  that    $\om^\pm,\ \Om^0$ satisfy hypothesis {\rm \ref{m0}}--{\rm \ref{m2}}, 
and that 
$$
\la< \la_1\big(\om^{0}\big).
$$
Let $\{\lambda_n\}$ be a sequence converging to  $\la .$
Then  any  {\rm (PS)$_{\lambda_n}$}  sequence, that is, a sequence $\{u_n\}\subset W_0^{1,p} (\Om)$ satisfying 
\begin{enumerate}
\item[$(PS_1)$]
$J_{\la_n }[u_n] \le C\, ,$ 
\item[$(PS_2)$] 
$|\langle J_{\lambda_n}'[u_n],  \psi\rangle |\le \e_n\,\|\psi\|$
where 
$\e_n \to 0$ as $n \to +\infty  ,$
\end{enumerate} 
is a bounded sequence.
\end{pro}

\begin{proof}  Let $\{u_n\}_{n\in \N} $ be a (PS)$_{\lambda_n}$ sequence in $W_0^{1,p}(\Om)$ and assume by contradiction that 
$\|u_n\|\rightarrow +\infty$.
The proof follows from  three claims.

\medskip

\noindent{\bf Claim 1.} {\it Let $v\in W_0^{1,p}(\Om)$ be the weak limit of $v_n=\frac{u_n}{\|u_n\|}$ and assume that $v_n\rightarrow v$, strongly   in $L^{p^*-1} (\Om)$ and a.e.   Then  $v=0$ a.e. }

\medskip

\noindent We divide the proof of this claim in two parts : (1.a)  $v= 0$ a.e. in $\Om\setminus \Om^0,$ and (1.b) $v= 0$ a.e. in $\Om^0$.\medskip

\noindent {\bf(1.a)}  \textit{$v\geq 0$, $v= 0$ a.e. in $\Om\setminus \Om^0,$  and $v_{|_{\om^0}}  \in W_0^{1,p}(\om^0).$}  

On the one hand, taking  $u_n^-$ as a test function in the condition $(PS_2),$  and since the supports of $u_n^+$, $u_n^-$ are disjoints, 
\begin{align*}
\|u_n^-\|^p&\leq \left|-\int_\Om |\na u_n^-|^p \, dx
- \int_\Om m(x)f (u_n^+)u_n^- \, dx \right| +\left|\int_\Om m(x)f(u_n^+)u_n^-\, dx\right|
\\[.3cm]
&\leq \e_n \|u_n^-\|  +C\|m\|_\infty |f(0)| \|u_n^-\|, 
\end{align*} 
so $\|u_n^-\|\leq D$   for some $D\geq0$,  and since $\|u_n\|\rightarrow +\infty$, then $v^-\equiv 0.$ 

Assume  now by contradiction that $v=v^+ \not=0$ a.e. in $\Om\setminus\Om^0$. Fix $\psi\in C_c^\infty (\Om)$.
By the reverse of the Lebesgue dominated convergence theorem, 
see for instance \cite[Theorem 4.9, p. 94]{Brezis_2011},
there exists $h\in L^1(\Om) $,   such that, up to a subsequence,
$$ 
\big(v_n^+(x)\big)^{p^*-1}\leq h(x), \;  \ \text{a.e.}\ x\in \Om,
$$ 
for all $n\in \N$. From the one hand,  by Lemma \ref{lem:conv} below, we have 
\begin{equation*}
m(\cdot) \frac{f(v_n^+\|u_n\|)}{f(\|u_n\|)}  \psi  \to  m(\cdot) (v^+)^q
\psi\qquad \hbox{strongly  in } L^1 (\Om).
\end{equation*}

On the other hand,  since Corollary \ref{cor:superlinear},  $f$ is $p$-superlinear at infinity,  implying  that
$$
\frac{1}{f(\|u_n\|)}\int_\Om |\na u_n|^{p-2}\na u_n \cdot\na \psi \, dx \rightarrow 0.
$$
Hence, using $(PS_2)$ for  the test function $\psi$, multiplying by $\frac{1}{f(\|u_n\|)}$ and passing to the limit we find
$$
\int_\Om m(x) (v^+)^q
\psi\, dx=0, \quad \forall \psi \in 
C_c^\infty (\Om).
$$
Hence  $0=v^+(x)= v(x)$, a.e.  $x\in\Om\setminus \Om^0$, a contradiction. In particular $v_{|_{\om^0}}  \in W_0^{1,p}(\om^0)$ as a consequence of \cite[Proposition 9.18]{Brezis_2011} since $\partial \om^0$ is of class $C^1$ by \ref{m2}. This   concludes the proof of   (1.a).\\

\noindent {\bf (1.b)} {\it $v= 0$ a.e. in $\Om^0.$} 

We claim 
 that  for any  $\psi \in C_c^{\infty}\big(\om^{0}\big)$  
\begin{equation}\label{rest}
\int_{\om^{0}}|\nabla v|^{p-2}\nabla v \cdot \nabla\psi \, dx = \la \int_{\om^{0}} v^{p-1}  \psi	\, dx\,. 
\end{equation}
Once this claim is  proved, using (1.a), thus, $v\ge 0$ and $v\in W_0^{1,p}\big(\om^{0}\big)$ and that  $\lambda<\lambda_1 \big(\om^{0}\big)$,
it must be $v\equiv 0$ a.e. in $\om^{0}$.  
Finally,  since $\partial \om^0$ is of measure $0$ by \ref{m2}, we  will  conclude {\bf (1.b)}.\medskip
 
Let us prove the claim.  For any 	$\psi \in W_0^{1,p}\big(\om^{0}\big)$
we have  by $(PS_2)$
\begin{equation}\label{rr:2}
\int_{\om^{0}}|\nabla v_n|^{p-2}\nabla v_n \cdot \nabla\psi \, dx = \la_n \int_{\om^{0}} (v_n^+)^{p-1}\psi	\, dx+o(1). 
\end{equation}
For any   $\chi\in C_c^1 (\om^0)$, $\chi\ge 0$, consider the test function $\psi =(v_n -v)\chi$ on $(PS_2)$. Since $v_n\to v$ strongly in  $L^p(\Om)$ we  have 
\begin{align*}
&\int_{\om^{0}} \chi\, |\nabla v_n|^{p-2}\nabla v_n\cdot\nabla(v_n-v)   \, dx =- \int_{\om^{0}}(v_n-v)|\nabla v_n|^{p-2}\nabla v_n \cdot \nabla \chi \, dx \\
&\qquad\qquad + \la_n \int_{\om^{0}} (v_n^+)^{p-1} (v_n  -v) \chi +o(1)	=o(1). 
\end{align*}
Moreover, since $v_n\rightharpoonup v$,
$$
\int_{\om^{0}} \chi\, |\nabla v|^{p-2}\nabla v\cdot\nabla(v_n-v)  \, dx=o(1).
$$
If we subtract the two above identities, we get 
$$
\int_{\om^{0}}\chi\, \left(|\nabla v_n|^{p-2}\nabla v_n-|\nabla v|^{p-2}\nabla v\right) \nabla(v_n-v)  \, dx=o(1).
$$ 
and using  the inequalities of Proposition \ref{S^+}
$$
 \int_{\om^0}\chi\,  |\na v_n- \nabla v|^{p}\, dx =o(1),
$$
for any $0\leq \chi \in C_c^1 (\om^0)$. 
Consequently,  
$v_n\to v$ strongly in $W_{loc}^{1,p}(\om^{0})$.
Coming back to \eqref{rr:2} and letting $n\to \infty$, we get \eqref{rest}.
\medskip

\noindent {\bf Claim 2} . {\it
$\lim_{n\rightarrow \infty}\int_\Om m(x)\frac{f(u_n^+)u_n^+}{\|u_n\|^p}\, dx =1$}.\\
Taking into account that $v\equiv 0$ a.e. in $\Om$, it follows from $(PS_2) $ applied to $\psi=u_n$ that  
$$
\int_\Om m(x)\frac{f(u_n^+)u_n^+}{\|u_n\|^p}\, dx 
=1 +o(1).
$$

\medskip

\noindent {\bf Claim 3} . {\it $\limsup_{n\rightarrow \infty}\int_\Om m^-(x)\frac{f(u_n^+)u_n^+}{\|u_n\|^p}\, dx <+\infty.$}\\
To prove this claim we will make use of hypothesis {\rm \ref{m2}}. It comes from this hypothesis that there exists $0\leq \psi\in C^1(\Om)$ such that $\psi=1$ on $\Om^-$ and 
$\psi=0$ on $\Om^+$. Taking $\phi=u_n \psi$ as a test function in $(PS_2)$ we get
\begin{align*}
&\int_\Om \left(\psi |\nabla v_n|^{p}+  v_n |\nabla v_n|^{p-2}\nabla v_n\cdot \nabla\psi \right)\, dx + \int_\Om m^-(x)\frac{f(u_n^+)u_n^+}{\|u_n\|^p}\, dx\\ &\qquad\qquad\qquad\qquad =\la_n \int_\Om (v_n^+)^p \psi \, dx +o(1)      
\end{align*}
from which it comes directly that the sequence $\left\{ \int_\Om m^-(x)\frac{f(u_n^+)u_n^+}{\|u_n\|^p}\, dx \right\}_{n\in \N} $ is bounded and the claim 3 is proved. 

\medskip
In order to achieve a contradiction,  we  use 
$(PS_1)$ and $(PS_2)$.
Multiplying by $1/p$ the inequality $(PS_2)$ applied to $u_n$  and adding $(PS_1)$ 
\begin{align*}
\int_\Om m(x)\bigg[\frac{1}{p}f(u_n^+)u_n^+- F(u_n^+)\bigg]\, dx 
&\le C+ \e_n\|u_n\|.
\end{align*} 
Consequently
\begin{align}\label{lim:1:J2:0}
\limsup_{n\rightarrow \infty}\int_\Om m(x)\bigg[\frac{\frac{1}{p}f(u_n^+)u_n^+- F(u_n^+)}{\|u_n\|^p}\bigg]\, dx\leq 0.
\end{align}
Using Theorem \ref{usfF} we have 
\begin{equation}
\label{F-sf}
\lim_{s\to +\infty}\, \frac{\frac{1}{p}sf(s)-F(s)}{sf(s)}=\frac{1}{p}-\frac{1}{q+1}=:\s>0.
\end{equation}
Fix $\e>0$.   Using \eqref{F-sf}, there exists $s_0>0$ such that 
$$(\sigma-\e)sf(s)\leq \frac{1}{p}sf(s)-F(s)\leq (\sigma+\e)sf(s)
\quad \forall s\geq s_0.$$
Splitting $\Omega$ into  $\{x\in\Omega : 0\leq u_n^+(x)\leq s_0\}$ and its complementary,  we conclude that there exists 
$C_\e\geq 0$ such that    for all $n$
\begin{align*}
-C_\e+\left(\s -\e\right)\int_\Om m^\pm(x)\,f(u_n^+)u_n^+\, dx
& \le
\int_\Om m^\pm(x)\left[\frac{1}{p}f(u_n^+)u_n^+- F(u_n^+)\right]\, dx\\  
&\le C_\e + \left(\s +\e\right)
\int_\Om m^\pm(x)\,f(u_n^+)u_n^+\, dx .
\end{align*}
Consequently, 
\begin{align*}
&\int_\Om\big(m^+(x)-m^-(x)\big)\left[\frac{1}{p}f(u_n^+)u_n^+- F(u_n^+)\right]\, dx\\
&\qquad\qquad \ge
-2C_\e +\left(\s -\e\right)\int_\Om \big(m^+(x)-m^-(x)\big)\,f(u_n^+)u_n^+\, dx\\
&\qquad\qquad\qquad -2\e\int_\Om m^-(x)\,f(u_n^+)u_n^+\, dx\,,
\end{align*}
and then
\begin{align*}
(\sigma-\e)\int_\Om m(x)\,f(u_n^+)u_n^+\, dx 
&\le 2C_\e+ \int_\Om m(x)\,\left[\frac{1}{p}f(u_n^+)u_n^+- F(u_n^+)\right]\, dx
\\
&\qquad 
+2\e\int_\Om m^-(x)\,f(u_n^+)u_n^+\, dx\,,
\end{align*}
so, using \eqref{lim:1:J2:0}
\begin{align*}
\lim_{n\rightarrow \infty}\int_\Om m(x)\,\frac{f(u_n^+)u_n^+}{\|u_n\|^p}\, dx
\le \frac{2\e}{\sigma-\e}\,\limsup_{n\rightarrow \infty}
\int_\Om m^-(x)\frac{f(u_n^+)u_n^+}{\|u_n\|^p}\, dx\,.
\end{align*}
Using  claim 3  and letting $\e \to 0$, we get a contradiction with claim 2
proving that any  Palais-Smale sequence   (PS)$_{\lambda_n}$ is bounded.  
\end{proof}

\begin{lem}\label{lem:conv}
Let $f\in C(0,+\infty)$ satisfy  
{\rm \ref{f3}}.
Let  $\{u_n\}$ be a sequence in $W_0^{1,p}(\Om)$, satisfying  $\|u_n\|\to +\infty$. 
Let $v$ be the weak limit of $v_n:=\frac{u_n}{\|u_n\|}$ in $W_0^{1,p}(\Om) $ converging to $v$ a.e.

Then  
\begin{equation}\label{cio}
\frac{f(u_n^+)}{f(\|u_n\|)} \to (v^+)^q, \qq{in}   L^1 (\Om).
\end{equation}
Similarly, if there exists  $h\in L^1(\Om)$ such that $|v_n|^{p^*}\leq h $ a.e. for all $n$ large, then 
\begin{equation}\label{cio2}
\frac{F(u_n^+)}{F(\|u_n\|)} \to (v^+)^{q+1}, \qq{in} L^1(\Om). 
\end{equation}
\end{lem}

\begin{proof}
Let us first prove the following claim: \\

{\em Claim }. {\it Let $\{\tau_n\}$ be a sequence of non-negative real numbers converging to some  $\tau_0\geq 0$ and let $\{s_n\}$ be a sequence of positive real numbers converging to $+\infty$.
Then 
$$
\lim_{n\to\infty}\frac{f(\tau_n s_n)}{f(s_n)} = \tau_0^q.
$$
} 

\medskip

The case $\tau_0>0$ follows from the representation Theorem \ref{rrvf} and from the uniform convergence result quoted in Theorem \ref{usfF}. We have
$$
\frac{f(\tau_n s_n)}{f(s_n)}=\tau_n^{q} \, \frac{L(\tau_n s_n)}{L(s_n)}\q{and} \lim_{n\to\infty} \frac{L(s_n\tau)}{L(s_n)}=1 \,,
$$
uniformly for any $\tau\in[a,b],$  where $L$ is defined in \eqref{f:q:L}--\eqref{L}.

\medskip

Observe that in the  case $\tau_0=0$ and  $\{\tau_n s_n\}_n$ converging  to some $c\in \R^+$ one have trivially the claim.
Assume $\tau_0=0$ and  $\tau_n s_n\to\infty.$
We use the integral representation of Theorem \ref{rrvf}.
Indeed, since $\displaystyle\lim_{s\to+\infty}\rho(s)=0,$ for any $\e >0$,  there exists  $t_\e>s_0$ such that 
\begin{equation*}
-\e \ln s
\leq \int_{t_\e}^{t_\e s}   \frac{\rho (t)}{t} \,dt 
\leq \e \ln  s
\qquad   \forall s\geq 1,
\end{equation*}
so, for  all $n$ large  enough
\begin{equation}\label{eps}
\tau_n^{\e}\le {\rm exp}\,\left(-\int_{\tau_n s_n}^{s_n} \frac{\rho(t)}{t} dt\right)\le \tau_n^{-\e}.
\end{equation}
Moreover
$$
\frac{f(\tau_n s_n)}{f(s_n)}=\tau_n^{q} \,\frac{L(\tau_n s_n)}{L(s_n)}=\tau_n^{q} \,\frac{\zeta(\tau_n s_n)}{\zeta(s_n)}{\rm exp}\,\left(-\int_{\tau_n s_n}^{s_n} \frac{\rho(t)}{t} dt\right).
$$
Fix  $\e>0$ small enough. 
Recalling that $\zeta$ is positive, continuous and have a strictly positive  finite limit at $+\infty$,  and since $\lim_{n\to\infty}s_n=\lim_{n\to\infty}s_n\tau_n=+\infty $ then for  all $n$ large  enough, 
\begin{equation}\label{bd:zeta}
\frac{1}{2}\leq \frac{\zeta(\tau_n s_n)}{\zeta(s_n)}\leq 2.
\end{equation}
Combining now \eqref{eps}-\eqref{bd:zeta}, for  all $n$ large  enough  we deduce
\begin{equation}\label{bd:fn}
\frac{1}{2}
\tau_n^{q+\e}\, 
\leq \frac{f(\tau_n s_n)}{f(s_n)}
\leq 
2
\tau_n^{q-\e} . 
\end{equation}
The claim  is proved in this  case letting $n\to+\infty$.\\

We now prove \eqref{cio}. 
Let $v_n=\frac{u_n}{\|u_n\|}$, then $v_n\rightharpoonup v$, strongly in $L^{p^*-1}(\Om)$. 
Since $\{v_n\}$ converges to $v$ strongly in $L^{p^*-1}(\Om)$,
there exists $h\in L^1(\Om)$ such that,  up to a subsequence, 
$
\forall n\in\N,\quad |v_n|^{p^*-1}\leq h \quad  \hbox{a.e.} 
$
and also $v_n\to v$ a.e. Using {\rm \ref{f3}}, and  that $u_n^+=v_n^+\|u_n\|$,  we have a.e. pointwise
$$
\left|\frac{f(u_n^+)}{f(\|u_n\|)}\right|\leq C( 1+|v_n^+|^{p^*-1} )\leq C(1+h)
.$$
Then the conclusion \eqref{cio} follows from the Lebesgue convergence theorem.

The  proof of \eqref{cio2} runs similarly.
\end{proof}
\medskip
\begin{thm}\label{th:comp:PS} 
Assume that	$f$  satisfies hypotheses  {\rm \ref{f1}}--{\rm \ref{f3}}, that $m \in C(\Omb)$,
satisfies hypothesis {\rm \ref{m0}}--{\rm \ref{m2}}. 
Let $\{\lambda_n\}$ be a sequence converging to  $\la.$ Let $\{u_n\}_{n\in \N} $ be a bounded {\rm (PS)$_{\lambda_n}$} sequence in $W_0^{1,p}(\Om)$.
Then there exists $u\in W_0^{1,p}(\Om)$ and  a subsequence, still denoted by $\{u_n\}_{n\in \N}$, such that
$u_n\to u $ in $ W_0^{1,p}(\Om).$
\end{thm}

\begin{rem}\rm
In the case $\la< \la_1(\om^0),$ it follows 
from Proposition \ref{pro:PS}  that any {\rm (PS)$_{\lambda_n}$} sequence $\{u_n\}_{n\in \N} $ is a bounded {\rm (PS)$_{\lambda_n}$} sequence. 
\end{rem}

\begin{proof}[Proof of Theorem \ref{th:comp:PS}]
Since $\{u_n\}_{n\in \N} $ is  bounded there exists a subsequence,  denoted by  $\{u_{n}\}_{n\in \N}$\,, and some $u\in W_0^{1,p}(\Om)$ such that 
\begin{align*}
& u_{n} \rightharpoonup u\ \ \text{weakly in}\ W_0^{1,p}(\Om),  \text{ strongly in } L^p (\Om)\\ 
&u_n\to u \ \text{ a.e.}
\end{align*}
By testing $(PS_2)$ against $\psi=u_n-u$, we can write 
\begin{align}\label{ror}
&\Bigg|\int_{\Om} |\na u_n|^{p-2} \na u_n\cdot (\na u_n-\na u)\,\, dx  -
\la_n\int_{\Om}(u_n^+)^{p-1}(u_n-u) \, dx\nonumber\\ &\qquad\qquad\qquad-\int_{\Om}m(x)\,f(u_n^+)\,(u_n-u)\, dx\Bigg|\le \e_n\| (u_n-u)\|.
\end{align}

\noindent {\bf Claim} . 
$\int_\Om |f(u_n^+)| |u_n-u|\, dx =o(1)$.

\medskip

 To prove the claim we use Lebesgue-Vitali convergence theorem, see \cite[Theorem 4.5.4, p. 268]{Bogachev}. An alternative proof can be done using Orlicz spaces when $f$ is increasing, see \cite[Proposition 3.1 and Theorem A.2]{Cuesta_Pardo_MN_2024}. 
 
Consider the sequence 
$$
h_n:=f(u_n^+)|u_n-u|, 
$$
which belongs to $L^1(\Om)$.  Since $|u_n -u|\to 0$ and $f(u_n^+)\to f(u^+)$ a.e., as  $f$ is continuous, then, $h_n\to 0$ a.e. and so in measure, see \cite[Theorem 2.2.3, p.111]{Bogachev}. In order to  prove that $h_n\to 0$ in $ L^1(\Om)$, we apply Lebesgue-Vitali convergence theorem 
(c.f \cite[Theorem 4.5.4, p. 268]{Bogachev}. It is then enough to prove that the sequence $\{h_n\}_n$ is {\it uniformly absolutely continuous}, that is,
\begin{equation}\label{uniint}\forall\eps>0\  \exists C_\e>0\text{ s.t.  }   \int_\omega|h_n|\, dx <\eps \quad \forall n\in\N, \; \forall \omega \subset \Omega, \;  |\omega|<C_\eps.\end{equation}
Let  $D:=\|u\|+\sup_n \|u_n\|<+\infty$ 
 and fix $\eps>0$.
By  {\rm \ref{f1}} there exists $C>0$ such that 
\begin{equation}\label{1}
f(s)\leq C(1+s^{p^*-1})\quad \hbox{ for all } s\geq 0 .\end{equation}
and  there exists $t_1=t_1(\eps)>0$  such that 
\begin{equation}
    \label{2}
    0\leq f(t)\leq \eps t^{p^*-1} \quad \hbox{ for all } t\geq t_1.\end{equation}
For any $n\in \N$ and  for any $\om\subset \Om$ we have  
\begin{align*}
\int_{\om}|h_n|\, dx &\le 
\int_{\om} |f(u_n^+)|\big(|u_n|+|u|\big)\, dx =
\int_{ \{ u_n\geq t_1\}\cap\om} |f(u_n^+)|\big(|u_n|+|u|\big)\, dx \\
&\qquad + \int_{ \{ u_n<t_1\}\cap\om} |f(u_n^+)|\big(|u_n|+|u|\big)\, dx=:I_1+I_2;
\end{align*}
and the estimates 
$$
\begin{array}{lll}
I_1&\leq \eps \int_\Om (u_n^+)^{p^*-1}(|u_n|+|u|) \, dx \leq C_1 \eps&\quad \hbox{by } \eqref{2} , \\
\\
I_2& \leq C(t_1^{p^*-1} +1)\int_{\om} (|u_n| + |u|)\, dx  &\quad \hbox{by } \eqref{1} ,\\[.3cm]
&\leq C(t_1^{p^*-1} +1) (\|u_n\|_{p^*}+\|u\|_{p^*})\,  |\om|^{1-1/p^*}  &\quad (\hbox{using Holder's ineq.})\\[.3cm]
&=: C_2(\eps)\,  |\om|^{1-1/p^*},
\end{array}
$$
where $C_1$   is a positive constant depending only  on  $D$  and $|\Om|$, 
 $C$ is a positive constant depending only  on $f$,
and  $C_2(\eps) $ is  a  positive constant depending on $C$, $|\Omega|$, $D$ and $t_1=t_1(\eps)$.
Thus,   if  $|\om|\leq  C_\eps:=\left(\frac{\eps}{C_2(\eps)}\right)^{\frac{p^*}{p^*-1}}$  then $$\int_\om |h_n|dx\leq (C_1 +1)\eps$$ and claim  \eqref{uniint} is proved.

\medskip
Now, we observe that since $u_n\wto u$  in $W_0^{1,p}(\Om)$,  
$$
\int_{\Om} |\na u|^{p-2} \na u\cdot (\na u_n-\na u)\,\, dx\to 0\qq{as}n\to\infty.
$$
and then,  coming back to \eqref{ror}, using the claim and Proposition \ref{S^+} we get that $u_n\to u$ strongly in $W_0^{1,p}(\Om)$.
\end{proof}

\subsection{An existence result for \texorpdfstring{$\la\leq \la_1$}{}}

The next theorem provides a solution to   \eqref{eq:ell:pb}$_\la$  for $\la\leq \la_1$ and $f$ slightly subcritical, by using the well known  Mountain--Pass Theorem when $\la<\la_1$, and a subtle study on the nature of local minimum due to de Figueiredo when $\la=\la_1$, see \cite[Theorem 5.10]{deFigueiredo}. 

\begin{thm}\label{th:la<la1}
Assume  that the nonlinearity $f$  satisfies  {\rm \ref{f1}}--{\rm \ref{f3}}, and also  {\rm \ref{H5}}-{\rm \ref{H6}}. Assume that the weight $m \in C(\Omb)$  
satisfies  hypothesis {\rm \ref{m0}}--{\rm \ref{m2}}.

Then,  the boundary value problem   \eqref{eq:ell:pb}$_\la$  has at least one   positive solution  $u\in C^{1,\mu}(\overline{\Om})$ for any $\la\leq \la_1$.
\end{thm}

\begin{proof}
We  verify the  hypothesis of the Mountain--Pass Theorem, see 
\cite[Theorem 2.1]{Ambrosetti_Rabinowitz}. We recall that  the  functional $J_\la:W_0^{1,p}(\Om)\to \R$ is $C^1$; also the \textrm{ (PS)$_{\lambda}$} condition is satisfied, see Proposition \ref{pro:PS} since $\lambda_1<\lambda_1\big(\om^{0}\big)$.
Clearly $J_\la[0]=0$.\\

\noindent {\bf Claim 1}.  There exists $v_0\in W_0^{1,p}(\Om)$ such that $J_\la [v_0]<0$. 

\medskip

To prove this claim,  fix some element $0\le u_0\leq 1$ in $ W_{0}^{1,p}(\Om)\cap C(\overline{\Omega})$, $ u_0\ge 0$, and $u\ne 0$ in $\Om^+$, and $u_0\equiv 0$ in $\Om^-$.

Using that $m(x)\,F\big(tu_0(x)\big)=m^+(x)\,F\big(tu_0(x)\big)$, and that
$F$ is a function regularly varying at infinity of order $q+1$, see \eqref{reg:var:F} in Theorem \ref{usfF}, we deduce  the pointwise limit
$$
\lim_{t\to\infty} m(x)\,\frac{F\big(tu_0(x)\big)}{F(t)}=m^+(x)\,\big(u_0(x)\big)^{q+1},
$$
and from  \eqref{reg:var:F} we have for all $t\geq 0$ 
$$
\left|m\,\frac{F\big(tu_0\big)}{F(t)}\right|\leq C(1+|u_0|^{q+1}) \in L^1(\Om).$$
Corollary \ref{cor:superlinear}  and Lebesgue convergence Theorem imply that
\begin{align*}
 \lim_{t\to\infty} \int_{\Om} m(x)\,\frac{F(tu_0)}{t^p}&=\lim_{t\to\infty} \frac{F(t)}{t^p}\int_{\Om} m(x)\,\frac{F(tu_0)}{F(t)} \\
 &= \left(\int_{\Om} m^+(x)\,u_0^{q+1}\right)\lim_{t\to\infty}\frac{F(t)}{t^p}= 
+\infty .   
\end{align*}

Obviously, there exists a positive constant $C>0$ such that
\begin{align*}
J_\la[tu_0]
&= \frac{t^p}{p}\|u_0\|^p -t^p\frac{\la}{p}\|u_0\|_{L^p(\Om)}^p- \int_{\Om} m(x) F(tu_0)\\
&\le Ct^p  -\int_{\Om} m^+(x) F(tu_0)
\end{align*}
hence,
$$
\frac{J_\la[tu_0]}{t^p}\leq C - \int_{\Om} m^+(x)\,\frac{F(tu_0)}{t^p}\to -\infty,\qt{as}t\to +\infty,
$$\medskip 
so there exists $t_0>0$ such that $J_\la[t_0u_0]<0$. Finally, choose $v_0=t_0u_0$ to conclude.

\medskip 

Next we study the behaviour of $J_\la$ near 0
and we distinguish two cases.
\smallskip 

\noindent {\bf Case $\lambda<\lambda_1$}.  

\medskip
\noindent {\bf Claim 2.} 
There exists $D_1>0$ and there exists $r_0>0$  such that 
$$
J_\la[u]\geq D_1\quad \forall u\in W_0^{1,p}(\Om), \quad \|u\|=r_0.
$$
Let  $u\in W_0^{1,p}(\Om)$ with $\|u\|=r_0$, for  $r_0>0$ to be selected below. Then, 
\begin{equation}\label{J:11:lambda}
J_\la[u]=\frac{r_0^p}{p}-\frac{\la}{p}\int_{\Om} (u^+)^{p}\, dx-\int_{\Om}\,m(x)F(u^+) \,\, \, dx\,.
\end{equation}
From hypothesis  {\rm \ref{f1}} and {\rm \ref{H5}}  we have
$$
\left|\int_{\Om}\,m(x)F(u^+) \,\, \, dx\right|\le 
\e r_0^p+Cr_0^{p^*}.
$$
Fix $0<\e\le \frac{1}{2}\left( 1-\frac{\la}{\la_1} \right)$. In view of \eqref{J:11:lambda}, and thanks to the Poincaré inequality we get
\begin{equation}\label{def:D1}
J_\la[u]\ge \frac{1}{2p}\left( 1-\frac{\la}{\la_1} \right)\, r_0^p -Cr_0^{p^*}\ge C_1 r_0^p-Cr_0^{p^*}=:D_1 
\end{equation}
Since  $\lambda<\la_1$ and that   $p^*>p$, taking $\e$ small  enough   we get that $D_1>0$ by choosing  $r_0=r_0(\la,  p,p^*,f)>0$ small enough and such that $r_0< \|t_0u_0\| $.\\

Let us  denote
$$
\Gamma :=\big\{ \mathbf{g}\in C\big([0,1];W_{0}^{1,p}(\Om)\big) :   \,\mathbf{g} (0)=0,\ \mathbf{g}(1)=v_0\big\},
$$
 and 
$$
c:=\inf_{\mathbf{g}\in \Gamma } \max_{0\leq t\leq 1} J_\la[\mathbf{g}(t)]
$$
so $c\geq D_1>0$, where  $D_1 >0 $ is defined  in \eqref{def:D1}. Since all the hypothesis of the Mountain--Pass Theorem are accomplished, $c>0$ is  a critical value of 
$J_\la$
and we get a nontrivial solution of \eqref{eq:ell:pb}$_{\la}$ for $\la<\la_1$.

\medskip

\noindent {\bf   Case $\lambda=\lambda_1$}. 
\medskip 

\noindent {\bf Claim  3.} There exists $r_1>0$ such that 
$$J_{\la_1}[u]\geq  0,\quad \forall u\in W_0^{1,p}(\Om),\quad \|u\|\leq r_1.
$$

Assume by  contradiction that the claim is not true. Hence  there exists $\{u_n\}$ in $W_0^{1,p}(\Om)$, $0<\|u_n\|\to 0$ and $J_{\la_1}[u_n]< 0$. Take $v_n=\frac{u_n}{\|u_n\|} $ and let $v$ be such that $v_n \rightharpoonup v$ in $W_0^{1,p}(\Om)$, strongly in $L^{p^* -1}(\Om)$ and a.e..  On the one hand,   by dividing 
the inequality $J_{\la_1}[u_n] \leq 0$ by $\|u_n\|^p$, letting $n\to\infty$   and using  
 (i) of  Lemma \ref{g0} below, we get that 
\begin{align}\label{ineq}
\int_\Om |\nabla v^+|^p\, dx &\leq \int_\Om |\nabla v|^p \leq 1=\liminf_{n\to\infty} \int_\Om |\nabla v_n|^p\, dx   \nonumber\\
&\leq \liminf_{n\to\infty} \left( \la_1\int_\Om |v_n^+|^p\, dx +p\int_\Om m(x)\frac{F(u_n^+)}{\|u_n\|^p}\, dx\right)\\
&=\la_1\int_\Om |v^+|^p\, dx \nonumber
\end{align}
and therefore $0\le v^+=c\varphi_1$ for some $c>0$. Observe that $c=0$ gives a contradiction with the above inequality $1\le \la_1\int_\Om |v^+|^p\, dx$.  Note also that since $v^+=c\varphi_1>0$, then $v^-=0$. Since now we have $\int_\Om |\nabla v^+|^p\, dx =\la_1 \int_\Om |v^+|^p\, dx,$
it follows from \eqref{ineq} that there exists a subsequence, still denoted by $v_n$, such that
$$ 
\int_\Om |\nabla v|^p\, dx=\lim_{n\to\infty} \int_\Om |\nabla v_n|^p\, dx.
$$
Now, using \cite[Proposition 3.32]{Brezis_2011} and since $W_0^{1,p} (\Omega)$  is uniformly convex for $p>1$, see \cite[p. 96, 97]{Brezis_2011}, then $v_n\to v$ strongly in $W_0^{1,p}(\Om)$.

In particular, $v_n^+  \to v^+=v$ strongly in $W_0^{1,p}(\Om)$ as well.
In the other hand, by  using the variational characterisation of $\lambda_1$, {\rm \ref{H6}},  (ii) of Lemma \ref{g0} below and that $F(s)>0$ for $s>0$ small by {\rm \ref{H5}} we find 
$$
\liminf_{n\to \infty}\frac{J_{\la_1}[u_n]}{F(\|u_n\|)}\geq-\int_\Om m(x) (c\varphi_1)^{\eta+1}\, dx>0
$$
in contradiction with the fact that $J_{\la_1}[u_n]\leq 0$ for all $n$. 
We have just proved Claim 3, that is, that $0$ is a local minimum of $J_{\la_1}$.
\\

We now use the  \cite[Theorem 5.10]{deFigueiredo} on the nature of local minima. The following alternative holds: either (i)
there exists $0<r_0<r_1$   such that
$$D_1:=\inf\{ J_{\la_1}[u] : \|u\|=r_0\}>J_{\la_1}[0]=0$$
or (ii) for each  $0<\alpha<r_1$, the functional $J_{\la_1}$ has a local minimum at a point $u_\alpha\in W_0^{1,p}(\Om)$
with $\|u_\alpha\|=\alpha$  and $J_{\la_1}[u_\alpha]=J_{\la_1}[0]=0$.
 \medskip
 
If alternative (ii) holds we are done, as $u_\alpha$ is a non trivial critical point of $J_{\la_1}$. If alternative (i) prevails,    denoting $\Gamma $ and $c$ as above  for $\la=\la_1$,
we obtain that $c\geq D_1>0$, where  $D_1 >0 $ is now defined  in  alternative (i) when $\la=\la_1$. Since all the hypothesis of the Mountain--Pass Theorem are accomplished, $c>0$ is a critical value of 
$J_{\la_1}$
and we get a nontrivial solution of \eqref{eq:ell:pb}$_{\la_1}$.

Thus, we have proved  that for any $\la\le\la_1$, there exists $u\in W_0^{1,p} (\Om)$, $u\ge 0$, $u\ne 0$ such that $J'_\lambda[u]=0$ and so $u$ is a nontrivial weak solution to  \eqref{eq:ell:pb}$_\la$. 
By   the regularity results  and the  strong maximum principles  for quasilinear elliptic equations  we deduce that  $u> 0$  and  $u\in C^{1,\mu} (\overline{\Om})$, see Remark \ref{reguC1}.
\end{proof}

\medskip

\begin{lem}\label{g0}
Assume that $f\in C(\R^+)$, 
satisfies  {\rm \ref{f1}}, {\rm \ref{H5}}  and \eqref{r}.
Let $\{u_n\}$ be a sequence in $W_0^{1,p}(\Omega)$ such that $\|u_n\|\to 0$ and let  $v_n:=\frac{u_n}{\|u_n\|}$. Assume that there exists  $v\in W_0^{1,p}(\Om)$  such that, up to a subsequence $v_n\wto v$  in $W_0^{1,p}(\Om)$
and  a.e. in $\Omega$. Then 
\begin{enumerate}[label=\rm{(\roman*)}]
    \item $ \displaystyle\frac{F(u_n^+)}{\|u_n\|^p}\to 0 \quad \hbox{ in } 
 L^1(\Om)$; 
 \item if moreover $L^{\infty}(\Omega)\ni v>0$ a.e. in $\Om$, then 
$\displaystyle \frac{F(u_n^+)}{F(\|u_n\|)}\to v^{\eta+1} \hbox{ in } 
 L^1(\Om).$
\end{enumerate}
\end{lem}
\begin{proof} (i) 
 Let $\epsilon >0$. Since   \ref{H5} there exists $s_0>0$ such that
$$|F(s)\leq \epsilon s^p  \qq{for all} 0\leq s\leq s_0.$$
On the other hand, it follows from  \ref{f1} and \ref{H5} that 
$$
\forall s\geq 0, \quad |F(s)|\leq A(s^p+s^{p^*})
$$
and then 
$$
\forall s\geq s_0, \quad |F(s)|\leq A(1+s_0^{p-p^*})s^{p^*}.
$$
Hence, 
\begin{align*}
\int_\Om \frac{F(u_n^+)}{\|u_n\|^p}\, dx &\leq \epsilon 
\int_{\{u_n\leq s_0\}} \frac{(u_n^+)^p}{\|u_n\|^p}\, dx +
\int_{\{u_n> s_0\}}
\frac{A(1+s_0^{p-p^*})(u_n^+)^{p^*}}{\|u_n\|^p}\, dx\\
\\ & \leq D_1\epsilon  +d_1A(1+s_0^{p-p^*})\|u_n\|^{p^*-p} 
\end{align*}
where  $d_1$, $d_1$  are positive constants coming from the Sobolev embeddings. Since $\|u_n\|\to 0$ as $n\to +\infty$, there exists $n_0\in\N$ depending on $\epsilon$ and $s_0$ such that $$\int_\Om \frac{F(u_n^+)}{\|u_n\|^p}\, dx\leq (d_1+1)\epsilon \qq{for all}  n\geq n_0,$$
and (i) follows.

\medskip

(ii) Consider the function $G(s):= F(\frac{1}{s}), s>0$ and  observe that $G\in RV_{-\eta-1}$ by \eqref{r}. By considering  the sequences  $s_n=1/\|u_n\|\to\infty$ and $\tau_n:=v_n^+(x)\to v(x):=\tau_0>0$, the  claim  of  the proof of  Lemma \ref{lem:conv} gives the pointwise limit $$
\lim_{n\to \infty}\frac{F\big(v_n^+ (x)\|u_n\|\big)}{F(\|u_n\|)}=v(x)^{\eta+1} \qq{a.e.}
$$
Reasoning as in \eqref{bd:fn}, fixing $\epsilon >0$, we have 
for  all $n$ large  enough,
\begin{equation*}
\frac{1}{2}\, \tau_0^{-r-1+\e}
\leq \frac{G(\tau_n^{-1} s_n)}{G(s_n)}
\leq 2\, \tau_0^{-r-1-\e} \qq{if}\tau_0<1, 
\end{equation*}
and 
\begin{equation*}
\frac{1}{2}\,  \tau_n^{-r-1-\e}
\leq \frac{G(\tau_n^{-1} s_n)}{G(s_n)}
\leq 
2\,  \tau_n^{-r-1+\e}  \qq{if}\tau_0>1. 
\end{equation*}
For a.e.  $x\in\Om$ we get 
\begin{equation*}
\frac{1}{2}\, v(x)^{\eta+1+\e}
\leq \frac{F\big(v_n^+ (x)\|u_n\|\big)}{F(\|u_n\|)}
\leq 2\, v(x)^{\eta+1-\e} \q{if}v(x)<1,
\end{equation*}
and 
\begin{equation*}
\frac{1}{2}\,  v(x)^{\eta+1-\e}
\leq \frac{F\big(v_n^+ (x)\|u_n\|\big)}{F(\|u_n\|)}
\leq 
2\,  v(x)^{\eta+1+\e}  \q{if}v(x)>1. 
\end{equation*}
Since $v\in W_0^{1,p}(\Om)\cap  L^\infty(\Om)$, we conclude again using Lebesgue dominated convergence theorem.
\end{proof}

\section{Proof of Theorem \ref{th:ex}}
\label{sec:proof}
The proof  of Theorem \ref{th:ex} for $\lambda \leq \lambda_1$ has been done in Theorem \ref{th:la<la1}. 

Furthermore, we get that   there exists an unbounded  closed and connected set $ \mathscr{C}^+$ of  positive  solutions to \eqref{eq:ell:pb} bifurcating from the trivial solution set at the bifurcation point $(\la,u)=(\la_1,0)$.

Let us define 
\begin{equation*}
\La:=\sup\big\{\la>0 : \,  \eqref{eq:ell:pb}_\la \ \text{admits a positive solution} \big\}.
\end{equation*}
Under condition {\rm \ref{H6}}, \eqref{a:vf1:<0} holds,   and the bifurcation is on the right (see Theorem \ref{th:ex2}), in other words,  $\La>\la_1$.
Moreover, by Proposition \ref{pro:non:exist}, $\La$ is upper bounded by \eqref{La}.\medskip

We will use the following order relations:
\begin{defi}\label{def:<}\textrm
1. Let $f,g:\Om\to \R$ be measurable  functions. We will denote by $f\prec g$  if for any compact $K\subset \Om$ there exists $\varepsilon >0$ such that $f(x)+\varepsilon < g(x)$ a.e $x\in K$.\\
2. Let $\Omega$ be a bounded smooth domain of $\R^N$, and   $u,v \in C_0^{1}(\overline{\Om})$, we will denote by $u\ll v$   if $v-u\in \mathscr{P}^+$ (see Definition \ref{def:P+}), in other words, if 
$u<v$ in $\Om$ and $\frac{\partial u}{\partial \nu} > \frac{\partial v}{\partial \nu}$ on $\partial\Om$. Here $\nu=\nu(x)$ stands for the outward normal at the point $x\in\partial\Om$.
\end{defi}

The applicability of those definitions can be seen in the Strong Comparison Principle for positive solutions, see Theorem \ref{th:scp}.

Recalling  the functional $J_\la$ defined in  \eqref{def:J:lambda}, let us prove the following   auxiliary result.

\begin{pro}
\label{pro:sol:min}
Assume that hypothesis {\rm \ref{f1}}--{\rm \ref{f3}} and {\rm \ref{H5}}--{\rm \ref{H6}} hold.
Assume also that $m\in C(\overline{\Om})$ satisfies {\rm \ref{m0}}--{\rm \ref{m2}}.

Then for each $\la\in(\la_1,\La)$,  problem \eqref{eq:ell:pb}$_\la$ admits a  positive solution 
$u_\la$  contained in  an interval 
$$
\mathcal{A}_\la:=\{u\in W_0^{1,p}(\Om):\underline{u}_\la\leq u\leq \overline{u}_\la\},
$$
where $\underline{u}_\la$, $\overline{u}_\la$ are appropriated sub- and super-solutions. Furthermore,  $u_\la$ is a local minimum of  the functional $J_\la$ and $J[u_\la]<0$.
\end{pro}

\begin{proof}
\noindent {\bf  Step 1}. {\it Existence 
of a pair of ordered sub- and super-solutions, for each   fixed $\la\in(\la_1,\La)$.}

Fixing $\la\in(\la_1,\La)$, we simplify the notation $\underline{u}=\underline{u}_\la,\ \overline{u}=\overline{u}_\la$. By definition of $\La$, there exists a $\la_0\in(\la,\La)$ such that the problem \eqref{eq:ell:pb}$_{\la_0}$ admits a positive solution $u_0$. It is easy to verify that $\overline{u}=u_0\gg 0$ is a super-solution to \eqref{eq:ell:pb}$_\la$. 
Next we prove that  for every $\de>0$  small enough,  with $\de=\de_\la$,
the function $\underline{u}=\de \vf_1$ is a subsolution for \eqref{eq:ell:pb}$_\la$ whenever $\la>\la_1$. Indeed, by {\rm \ref{H5}} there exists $t_1>0$ such that 
\begin{equation}\label{rrr}
|f(s)|\leq \left(\frac{\lambda-\lambda_1}{2\|m\|_\infty}\right)s^{p-1}, \quad\quad  \forall\,  0\leq s\leq t_1.
\end{equation}
So
\begin{equation}\label{rrr2}
|F(s)|\leq \left(\frac{\lambda-\lambda_1}{2p\|m\|_\infty}\right)s^{p}, \quad\quad  \forall\,  0\leq s\leq t_1.
\end{equation}
Let $\de >0$ be such that 
$\de<\min\{t_1, \inf_{\Om}u_0/\varphi_1\}$.
Thus  for any $\psi\in W_0^{1,p}(\Om)$, $\psi>0$ within $\Om$ we deduce
\begin{align}
\label{def:J:la:de:vf}
J_\la[\de\vf_1]
&\ = -\frac{(\la-\la_1)\de^{p}}{p} \int_\Om  (\vf_1)^p \, dx -\int_\Om m(x)F(\de\vf_1) \, dx\nonumber\\
&\  \le -\left(\frac{\lambda-\lambda_1}{2p}\right) \int_\Om (\de\vf_1)^{p}\, dx  < 0.
\end{align}
Moreover
\begin{align}
\label{def:J':la:de:vf}
\langle J_\la'[\de\vf_1], \psi \rangle
&  = -(\la-\la_1)\de^{p-1} \int_\Om  \varphi_1^{p-1}\psi \, dx -\int_\Om m(x)f(\de\vf_1) \psi\, dx\nonumber\\
&  \le 
-\left(\frac{\lambda-\lambda_1}{2}\right) \int_\Om (\de\vf_1)^{p-1}\psi\, dx  < 0.\nonumber 
\end{align}
This allows us to take $\underline{u}:=\de \varphi_1$ as a subsolution for \eqref{eq:ell:pb}$_\la$ with $\underline{u}<\overline{u}:=u_0$. The sub- and supersolution method now guarantees the existence of a positive solution $u_\lambda$ to \eqref{eq:ell:pb}$_\la$, on the interval $\mathcal{A}_\la$.  Moreover, $u_\lambda$ achieves the infimum of the functional $J_\lambda $  defined by \eqref{def:J:lambda},   cf. \cite[Proposition 3.1]{deFigueiredo_Gossez_Ubilla}.

\medskip

\noindent {\bf Step 2}.  {\it We claim that $\underline{u}\ll u_\lambda\ll \overline{u}
$} . \medskip

We start by proving that $u_\lambda\ll \overline{u}$. Recall that $\overline{u}=u_0$  is a solution of  \eqref{eq:ell:pb}$_{\lambda_0}$  for some fixed $\lambda_0>\lambda$. Let $c_1>0$  be given by \eqref{c1} for $s\in [0,2\max_{\Omb} u_0]$.
We have
$$
\left\{ \begin{array}{ll}
-\Delta_p u_\lambda +g(x, u_\lambda)=h_1; \quad -\Delta_p u_0 +g(x,u_0)=h_2 \quad  \hbox{ in } \Om,\\
u_\la=u_0=0 \quad \hbox{ on } \partial\Om
\end{array}\right.
$$
with $g(x,s):= m^- (x) \big[f( s) +c_1 s^{p-1}\big] +c_1m^+ (x) s^{p-1}$, and
\begin{align*}
h_1(x)&:= \lambda  u_\la^{p-1} +m^+ (x) \big[f( u_\lambda) +c_1 u_\lambda^{p-1}\big] +c_1m^- (x) u_\lambda^{p-1},\\ 
h_2(x)&:= \lambda_0  u_0^{p-1} +m^+ (x) \big[f( u_0) +c_1u_0^{p-1}\big] +c_1m^-  (x)u_0^{p-1}.
\end{align*}
By hypothesis \ref{f2}, $g$ is non decreasing 
in $\overline{\Om}\times [0, \max_{\overline{\Om}}u_0]$. Since  $f$ is continuous we can extend $g$ in a continuous way in $\overline{\Om}\times \R^+$, and $g(x,\cdot)$ in a non-decreasing way.

Moreover, $u_0\ge u_\la$, and since $u_0\gg 0$, it is easy to see that   $u_0\succ 0$, and then
\begin{align*}
h_2-h_1&= (\lambda_0-\lambda)u_0^{p-1} +\lambda \big(u_0^{p-1}- u_\la^{p-1}\big) + c_1m^-    \big(u_0^{p-1}-u_\lambda^{p-1}\big)\\
&\quad +m^+  \Big[ \big(f( u_0) +c_1u_0^{p-1}) - \big(f( u_\lambda) +c_1 u_\lambda^{p-1}\big)\Big]
\geq (\lambda_0-\lambda)u_0^{p-1}\succ 0.
\end{align*}
Using Theorem \ref{th:scp} we conclude that $u_\lambda\ll u_0$.
The proof that $\underline{u}=\delta\varphi_1\ll u_\la$ runs similarly, as one have 
$$
\left\{ 
\begin{array}{ll}
-\Delta_p (\delta \varphi_1) +g(x,\delta \varphi_1)=h_0&\quad  \hbox{ on } \Om,\\
\delta \varphi_1=0 &\quad \hbox{ in } \partial\Om,
\end{array}
\right.
$$
with 
$h_0:=\lambda_1 (\delta \varphi_1)^{p-1}+g(\cdot, \delta \varphi_1)$.
Obviously, $u_\la\ge \de\vf_1$. By adding $\pm m^+f(\delta\vf_1)$, using \eqref{c1}, in the first inequality, and  \eqref{rrr} in the last equality, we have
\begin{align*}
h_1-h_0
&= \la u_\la^{p-1}-\la_1(\delta\vf_1)^{p-1}  +m^+  \big[f(u_\la) +c_1 u_\la^{p-1} -c_1(\delta\vf_1)^{p-1}\big] \\
& \qquad +m^-  \Big[c_1 \Big(u_\la^{p-1}-(\delta\vf_1)^{p-1}\Big) - f(\delta\vf_1) \Big]\\
&\geq  
(\lambda-\lambda_1) (\delta\vf_1)^{p-1} +mf(\delta\vf_1)\\
&\geq \frac{\lambda-\lambda_1}{2} (\delta\vf_1)^{p-1} \succ 0 .
\end{align*}

\noindent {\bf Step 3}.  {\it   $u_\la$ is a local minimum of $J_\la$ and $J_\la(u_\la)< 0$.}\medskip

First, observe that the set 
$$
{\mathcal B}_\lambda:=\{v\in C_0^1(\overline{\Om}) : \underline{u}\ll v\ll \overline{u}\} 
$$ 
is open on the $C_0^1(\overline{\Om})$-topology. 
Moreover, $ u_\lambda\in {\mathcal B}_\lambda,$ and, since ${\mathcal B}_\lambda \subset {\mathcal A}_\lambda$,  we have that $u_\lambda$ is  a local minimum of $J_\lambda$ on ${\mathcal B}_\lambda $.  Therefore, there exists  $r_0>0$ sufficiently small  such that 
\begin{equation}\label{u:la:min}
J_\lambda  [u_\lambda]\leq J_\lambda [u_\lambda+ w] \quad \forall w\in C_0^1(\overline{\Om}), \quad \|w\|_{C_0^1(\overline{\Om})}\leq  r_0.    
\end{equation}
Brézis and Nirenberg  proved that a local minimizer on the $C^1$ topology must be a local minimizer on the $H^1$ topology, see \cite{Brezis_Nirenberg}. In \cite[Proposition 3.9]{deFigueiredo_Gossez_Ubilla}, de Figueiredo, Gossez and Ubilla
generalized that result for the $p$-Laplacian case. 
Using this generalisation, we conclude that $u_\lambda$ is a local minimizer of $J_\lambda$ with respect to the  $W_0^{1,p}(\Om)$-topology.
\smallskip

Finally, using \eqref{def:J:la:de:vf},
$J_{\la}[u_\la]\leq J_{\la} [\delta \varphi_1]<0,$ which ends the Proof of Proposition \ref{pro:sol:min}.
\end{proof}

\smallskip

\begin{proof}[Proof of Theorem \ref{th:ex} for $\lambda \in (\lambda_1,\Lambda)$]

A second solution  to \eqref{eq:ell:pb}$_\la$ will follow through variational arguments \cite[Theorem 5.11]{deFigueiredo} and the Mountain--Pass procedure shown below.

First, we  apply again \cite[Theorem 5.10]{deFigueiredo} on the nature of local minima  $u_\la$ found in Proposition \ref{pro:sol:min}. Let $r_0>0$ be as in \eqref{u:la:min}.
Thus,  according to \cite[Theorem 5.10]{deFigueiredo}, either

\smallskip

(i) 
for all $\e\in(0,r_0]$, there exists $\e_0\in(0,\e)$, such that  $\inf\big\{J_\la [u] : \|u-u_\la\| = \e_0 \big\} >J_\la [u_\la],$ or else

\smallskip 

(ii) for each $\de >0$ small enough, there exists $u_\de \in W_0^{1,p}(\Om)$ such that  $J_\la$ has a local minimum at a point $u_\de$  with $\|u_\de-u_\la\|=\de$ and $J_\la [u_\de]=J_\la [u_\la].$ 
\smallskip

Let us assume that (i) holds,  since otherwise  case (ii) implies the existence of a second solution $u_\de\not=0$,  since $J_\la [u_\de]=J_\la [u_\la]<0$, which is then a positive solution of \eqref{eq:ell:pb}$_\la$.
Remember that nontrivial critical points of $J_\la$ are positive solutions of \eqref{eq:ell:pb}$_\la$, see Remark \ref{sol:+}. \\

Consider  the functional $I_\la : W_0^{1,p} (\Om)  \to \R$ given by $I_\la[v]=J_\la[u_\la+v]-J_\la[u_\la]$, more specifically
\begin{align*}
I_\la[v]&:=\frac{1}{p} \int_{\Om}\,  |\na (u_\la+v)|^p\, dx-\frac{\la}{p} \int_{\Om}\,   |(u_\la +v)^+|^p \, dx
-\frac{1}{p}\int_\Om |\nabla u_\la |^p\, dx\\
& +\frac{\la}{p} \int_\Om | u_\la |^p\, dx 
-\int_{\Om}\, m(x)\Big[F((u_\la +v)^+)-F(u_\la)\Big]\, dx\,.
\end{align*}

\noindent  We are going to prove that there exists a function $u_0$ such that $I_\la  [u_0]<0$. Take   $v_0\geq 0$, $v_0\not\equiv 0$ in $W_0^{1,p}( \Om^+)$,  $v_0=0$ in $\Om^-$.
Since 

\begin{align*}
\lim_{t\to+\infty} &\frac1{t^p} \left(\frac{1}{p} \int_\Om |\nabla (u_\la +tv_0)|^p\, dx -\frac{\la}{p} \int_\Om | u_\la +tv_0|^p\, dx  -\frac{1}{p} \int_\Om |\nabla u_\la |^p\, dx \right.\\
&\qquad +\left.\frac{\la}{p} \int_\Om | u_\la |^p\, dx +\int_{\Om}m^+(x)F(u_\la)dx\right) \\
&= \frac{1}{p} \int_\Om |\nabla v_0|^p\, dx -\frac{\la}{p} \int_\Om | v_0|^p\, dx,
\end{align*}
and $F(t)/t^p\to \infty$ as $t\to \infty,$ then
\begin{equation}\label{ee}
\lim_{t\to+\infty} \frac{I_\la [tv_0 ]}{F(t)} =- \lim_{t\to+\infty}\int_{\Om}m^+(x)\, \frac{F\big(u_\la(x) +tv_0(x)\big)}{F(t)}\, dx.
\end{equation}
Besides, using that $F$ is a function regularly varying at infinity of order $q+1$, see \eqref{reg:var:F} in Theorem \ref{usfF}, we deduce that   pointwise 
$$
\lim_{t\to\infty} \frac{F\big(u_\la (x) +tv_0(x)\big)}{F(t)}=\big(v_0(x)\big)^{q+1}.
$$
Furthermore, since  for all $t\geq 1$, $|\frac{u_\la}{t}+v_0|\leq |u_\la |+|v_0|\in L^{p^*}(\Om)$,
using a similar argument to prove \eqref{cio2} in Lemma \ref{lem:conv} we get that
 $$\lim_{t\to\infty} \frac{F\big( t(\frac{u_\la}{t}+ v_0)\big)}{F(t)}=v_0^{q+1}
\quad \hbox{ in } L^1 (\Om). 
$$
So, Lebesgue convergence Theorem imply  in \eqref{ee} that
$$
\lim_{t\to+\infty} \frac{I_\la [tv_0 ]}{F(t)} =-\int_{\Om}m^+(x)v_0^{q}\, dx<0 
$$
and we can get for $t_1>0$ large enough that $I_\la [t_1 v_0]<0$.
\medskip 

Finally, let us define
$$
\Gamma :=\{ \mathbf{g}\in C\big([0,1];W_{0}^{1,p}(\Om)\big) : \,\mathbf{g} (0)=0,\ \mathbf{g}(1)=t_1v_0\},
 $$
and  the value 
$$
c:=\inf_{\mathbf{g}\in \Gamma } \max_{0\leq t\leq 1} I_\la[\mathbf{g}(t)]<+\infty. 
$$

If follows from the fact that $u_\lambda$ is a  local minimum of $J_\lambda$, and that $I_\lambda (0)=0$, that $0$ is a  local minimum of $I_\lambda$.    We are assuming that alternative (i) prevails for $J_\lambda$.  Since for  any path $\mathbf{g}\in \Ga$ will meet the set  $\|v\|=\eps_0$, it comes that 
$$c\geq \min\{ I_\lambda (v) : \|v \|=\varepsilon_0\}>0. $$ 

Moreover, from \cite[Theorem 5.1]{deFigueiredo} we derive that there exists a sequence $\{v_n\}_{n\in\N}$ such that
$$
I_\la[v_n]\to c,\qquad \|I_\la'[v_n]\|\to 0.
$$
Thus, the sequence $w_n:=u_\la+v_n$ is a (PS) sequence for the original functional $J_\la.$ 
Since Theorem \ref{th:comp:PS}, if 
 $\lambda<\la_1( \om^0)$
there exists a subsequence, still denoted $\{v_n\}$, and there exists $v_\la\in W_0^{1,p}(\Om)$ such that $v_n\to  v_\la$ in $W_0^{1,p}(\Om).$ So $I_\la'[v_\la]=0$, $I_\la[v_\la]=c>0,$ $v_\la\ne 0$. Consequently, $w_\la:=u_\la+v_\la$ is a positive critical point of $J_\la,$ 
and thereby  $w_\la\ne u_\la$ is a second positive solution to \eqref{eq:ell:pb}$_\la$, see Remark \ref{sol:+}. 
\end{proof}

\smallskip

\begin{proof}[Proof of Theorem \ref{th:ex} \, for $\lambda=\Lambda$]

Let $\{\la_n\}\subset\R$ be a non-decreasing sequence converging to $\La$ with $\lambda_n\in (\lambda_1,\La)$.  

Proposition \ref{pro:sol:min}, ensures that for each
$\la\in(\la_1,\La)$,  there exists a  positive solution $u_\la$ to the problem \eqref{eq:ell:pb}$_\la$. Moreover, the solution $u_\la$ is contained in  an interval $\mathcal{A}_\la:=\{u\in W_0^{1,p}(\Om):\underline{u}_\la\leq u\leq \overline{u}_\la\}$, where $\underline{u}_\la$, $\overline{u}_\la$ are ordered sub- and super-solutions. Furthermore,  $u_\la$ is a local minimum of  the functional $J_\la$ and $J_\la[u_\la]<0$. For each $n$, pick into that family of positive solutions, $u_{n}=u_{\la_{n}}$ solving \eqref{eq:ell:pb}$_{\la_{n}}$. 

\bigskip

On the other hand, $J_{\lambda_n}'[u_n]=0$ for all $n$, then $\{u_n\}$ is, moreover, a (PS)$_{\la_n}$ sequence. Since  we are assuming that $\La<\la_1\big(\om^{0}\big)$,   hence the sequence   $\{u_n\}$  possesses a subsequence strongly convergent in $W_0^{1,p} (\Om)$ by Theorem \ref{th:comp:PS}. Denote by $u$ the limit of that converging subsequence,  still denoted by $\{u_n\}$.  Passing to the limit in the weak  equation of \eqref{eq:ell:pb}$_{\lambda_n}$ we get that $u\geq 0$ solves \eqref{eq:ell:pb}$_{\La}.$ It remains to prove that $u\not\equiv 0$.
Indeed,  if on the contrary $\|u_n\|\to 0$,  the sequence $w_n:=\frac{u_n}{\|u_n\|}\geq 0$ converges weakly to some $w\geq 0$ in $W_0^{1,p}(\Om)$ and strongly in $L^p (\Om)$.  Testing against $\psi=u_n$ the weak  equation of \eqref{eq:ell:pb}$_{\lambda_n}$ divided by $\|u_n\|_{}^{p}$ and going to $+\infty$ we find, using {\rm \ref{H5}},  that
$1=\La \int_\Om w^p\, dx$  	
so $w\not=0$. Also, testing against  $\psi=w_n-w$, dividing by $\|u_n\|^{p-1}$ and  using  {\rm \ref{H5}} and the $S^+$ property, c.f. Proposition \ref{S^+}, we prove that $w_n\to w$ strongly in $W_0^{1,p}(\Om)$. Thus, $w$ is a nonnegative nontrivial solution of the weak equation
$$
\int_\Om |\nabla w|^{p-2}\nabla w\cdot \nabla \psi \, dx =\Lambda \int_\Om w^{p-1}\psi \, dx
$$
which is impossible because  $\Lambda\not=\lambda_1$.
\end{proof}

\begin{appendix}
\section{Strong Comparison Principle for positive solutions}

\renewcommand{\thesection}{\Alph{section}}
\numberwithin{equation}{section}
\numberwithin{thm}{section}

Next, we give an strong comparison principle for positive solutions, useful in the proof of Theorem \ref{th:ex}, see section \ref{sec:proof}. We slightly generalise the result of \cite[Proposition 2.6]{Arcoya_Ruiz}. 

For the order notations $f\prec g$ and $u\ll v$ we recall Definition \ref{def:<}.

\begin{thm}\label{th:scp}
Let $p>1$,  $\Om $  a bounded   domain in $\R^N$ of class $C^1$, and let $g:\Om\times \R\to \R$ be a continuous function. 
Assume further that for all $x\in \Omega$, the function $s\to g(x,s)$ is non-decreasing. Let $f,h \in L^{\infty}(\Om)$ be such that $f\prec h$ and let  $u,v \in C_0^{1}(\overline{\Om })$, such  that   $0\ll v$, weakly satisfy
\begin{equation*}  
\begin{cases}
- \Delta_p u +g(x,u) =f, &\text{ \ \ in \ } \Om, \\
-\Delta_p v+g(x,v)=h, &\text{ \ \ in \ } \Om, \\
\qquad u  =v=0  , &\text{ \ \ on \ } \partial \Om .
\end{cases}
\end{equation*}
Then $u \ll v $. 
\end{thm}
\begin{proof}
One can reproduce with some minor changes the proof of \cite{Arcoya_Ruiz}, where the case $g(x,s)=\lambda |s|^{p-2}s$ with $\lambda\geq 0$ is considered. 
\end{proof}

\section{Vector inequalities and  the \texorpdfstring{$S^+$}{} property}
Let us recall in this section the following inequalities which are mostly used to prove the so-called $S^+$-property of the $p$-Laplacian on $W_0^{1,p}(\Om)$ when $1<p<+\infty$, see \cite{Lindqvist}.

\begin{lem}
Let  $a,b\in \R^N$ and  $p>1$. Then 
\medskip

\noindent (i) if $p\geq 2$,  
    $\langle |b|^{p-2}b -|a|^{p-2}a,  b-a\rangle \geq 2^{2-p}|a-b|^p $;
\medskip    
    
\noindent (ii) if $1<p<2$, \,    
   $\langle |b|^{p-2}b -|a|^{p-2}a, \, b-a\rangle \geq \dfrac{(p-1)|a-b|^2}{\left(1+|a|^2+|b|^2\right)^{1-\frac{p}{2}}}$ .
\end{lem}
As an application one can prove the so-called $S^+$ property of the $p$-laplacian, that  we presented here on a local version.
\begin{pro}\label{S^+}
Let  $\Om\subset\R^N$ be an open bounded set and  $p>1$. Let $u\in W_0^{1,p}(\Om)$,  $\{u_n\}_n $ be a sequence in $ W_0^{1,p}(\Om)$  and  $\psi\in C_c^{1}(\Om)$, $\psi\geq 0$. 
If 
$$  \lim_{n\to\infty}\int_\Om \psi \left( |\nabla u_n|^{p-2} \nabla u_n -|\nabla u|^{p-2}\nabla u\right) \cdot (\nabla u_n-\nabla u)\, dx =0$$
then 
$$\lim_{n\to\infty}\int_\Om \psi |\nabla u_n-\nabla u|^p\, dx=0.$$
\end{pro}

\section{Compactness for slightly subcritical nonlinearities}

Consider the quasilinear elliptic problem with parameter $\lambda$
\begin{equation*}
\left\{
\begin{array}{rll}
-\De_p u &=\ \la  |u|^{p-2}u+g(\la ,x,u), & \mbox{in } \Om,\\
u&=\ 0,  &\mbox{on } \partial \Om,
\end{array}
\right.    
\end{equation*}
where $g:\R\times\Omega\times  \R\to\R$ is
a Carath\'eodory function such that, for any bounded set $J\subset \R$ there exists some continuous  function  $\bar{f}:\R^+\to\R^+$ such that 
\begin{equation}\label{crois g}
|g(\la ,x, s)|\leq 
\bar{f}(|s|)\quad a.e. \, x\in \Omega, \quad \forall (\lambda,s)\in J\times \R,
\end{equation}
where $\bar{f}$ is slightly subcritical (i.e. $\bar{f}$ satisfies {\rm \ref{f1}} for $f=\bar{f}$).
\medskip

Let us define the following operators $N,S: W_0^{1,p}(\Omega)\to W^{-1,p'}(\Omega)$ and  $G :\R\times W_0^{1,p}(\Omega)\to W^{-1,p'}(\Omega)$ :
\begin{align*}
\langle N(u),\psi\rangle &=\int_\Omega |\nabla u|^{p-2}\nabla u\cdot \nabla \psi \, dx,
\\
\langle S(u),\psi\rangle &=\int_\Omega | u|^{p-2}u \psi \, dx,
\nonumber\\
\langle G(\lambda, u),\psi\rangle &=\int_\Omega g(\la ,x, u)\psi\, dx.\nonumber 
\end{align*}
Define the product space $E:=\R\times W_0^{1,p}(\Omega)$ endowed with the  norm $\|(\lambda, u)\|_E:=|\lambda|+\|u\|$.\\

In order to  reformulate Drabek's bifurcation theorem for the equation $N-\lambda S -G(\lambda,\cdot)=0$ we need to assure that $G$ is a {\em compact} operator. 
\begin{pro}\label{pro:compG}
Assume that $g$ satisfies the growth condition \eqref{crois g}
with $\bar{f}$   continuous and satisfying  {\rm \ref{f1}} for $f=\bar{f}$. 

Then for any bounded closed interval $I\subset \R$,  the operator $G:I\times W_0^{1,p}(\Om)\to W^{-1,p'}(\Omega)$ is compact.
\end{pro}

\begin{proof}
Let $\{\la_n\}_{n\in\N}$ be a sequence in $I$, and $\{u_n\}_{n\in\N}$ a bounded sequence of $W_0^{1,p}(\Om)$. Then, up to a subsequence, there exist $\lambda_0\in I$ and $u$ such that $\lambda_n\to\lambda_0$, $u_n \wto u$ weakly in $W_0^{1,p}(\Omega)$, strongly in $L^{p^*-1}$ a.e. 

Let us show that  the sequence $\{ G(\lambda_n, u_n)\}_{n\in\N}$ converges in $W^{-1,p'}(\Omega)$. 
For any $\psi\in W_0^{1,p}(\Omega)$,
$$
\left|\int_\Omega \big( g(\la_n,x, u_n)-g(\la_0 ,x, u)\big)\psi\, dx\right|\leq \|\psi\|_{p^*}\,\big\|g(\lambda_n, \cdot, u_n)-g(\lambda_0, \cdot, u)\big\|_{(p^*)'}
$$
Let $z_n:=|g(\lambda_n, \cdot,u_n)-g(\lambda_0, \cdot,u)|^{(p^*)'}$. In order to prove that $\|z_n\|_1\to 0$ we will made use of the {\it Lebesgue-Vitali convergence theorem}. First,  since $g(\cdot, x, \cdot)$ is continuous, the sequence $\{z_n\}$  converges 0 a.e. and therefore, it converges also to 0 in measure. Secondly, we claim that $\{z_n\}_n$ is   uniformly  absolutely continuous, that is, 
\begin{equation}\label{equii}
\forall \e>0 \; \exists \delta>0 \hbox{ s.t. } 
 \int_{\omega} z_n\, dx\leq \e, \quad \forall n\in \N, \quad  \forall \omega\subset \Omega, \ |\omega|<\delta.
\end{equation}
Trivially it is enough to prove \eqref{equii} for   all $n\in\N$ large enough.
\medskip
We write 
$$\int_\omega z_n\, dx \leq C\left( \int_\omega |g(\lambda_n, x, u_n)|^{(p^*)'}\, dx +\int_\omega |g(\lambda_0, x, u)|^{(p^*)'}\, dx\right)$$
for some $C>0$ independent of $n$.   By  \ref{f1} there exists  $t_1>0$ such that 
\begin{equation}\label{tt}
|g(\lambda, x, t)|\leq \e ^{1/(p^*)'}|t|^{p^*-1}\quad a.e.\ x\in\Omega,\; \forall \lambda\in I\subset\R, \, \forall |t|\geq t_1.\end{equation}
Let $\omega \subset \Omega $ be a measurable set. We write 
\begin{align*}
 \int_\omega |g(\lambda_n, x, u_n )|^{(p^*)'}\,dx 
 &= \int_{\{u_n>t_1\}\cap \omega } |g(\lambda_n, x,  u_n )|^{(p^*)'}\,dx \nonumber\\
 &\quad +\int_{\{u_n \leq t_1\}\cap\omega}|g(\lambda_n, x,  u_n )|^{(p^*)'}\,dx 
\nonumber\\& 
=: I_1+I_2.
\end{align*}
Using \ref{f1} again, there exists $C_1>0$ such that 
$$|g(\lambda, x, t)|\leq C_1(1+|t|^{p^*-1})\quad a.e.\ x\in\Omega,\; \forall \lambda\in I, \, \forall t\in \R$$
then, denoting still by $C_1$ a positive constant independent of $n$ and $\e $
$$I_2=\int_{\{u_n \leq t_1\}}|g(\lambda_n, x, u_n )|^{(p^*)'} \chi_{\omega}dx 
\leq C_1 (1+ |t_1|^{p^*-1}  )^{(p^*)'}  |\om| . $$
Thus, if $|\omega| \leq  \frac{\e }{C_1 (1+ |t_1|^{p^* -1}  )^{(p^*)'} }$ then 
$I_2\leq \e.$
Next, using  \eqref{tt} 
$$I_1\leq \e \int_\omega |u_n|^{p^*}dx \leq D\e $$
where $D=\sup_n \|u_n\|_{p^*}^{p^*}$. 
Similarly  we can prove $$\int_\omega |g(\lambda_0, x, u)|^{(p^*)'}\, dx\leq C \e$$
for some $C>0$ and the claim \eqref{equii} is proved.
\end{proof}

Now we can prove the following bifurcation theorem :
\begin{thm}\label{pro:bif drabek}
Let $g$ satisfies the growth condition \eqref{crois g} with $\bar{f}(0)=0$, 
$\bar{f}$   continuous, 
slightly subcritical at infinity,  see  {\rm \ref{f1}}, satisfying 
also
\begin{equation*}
 \lim_{s\to 0}\frac{\bar{f}(s)}{s^{p-1}}=0.  
\end{equation*}
Then there exist  a maximal closed connected  set ${\mathscr C}$ of ${\mathscr S}$ containing $(\lambda_1,0)$  such that  either
\begin{enumerate}
\item[\rm (i)] ${\mathscr C}$ is unbounded on $E$, or else
\item[\rm (ii)] ${\mathscr C}$ contains  a point $(\tilde{\lambda},0)$ where $\tilde{\la}$ is an eigenvalue of $(-\Delta_p, W_0^{1,p}(\Om))$, with $\tilde{\lambda}\not=\lambda_1$.
\end{enumerate}
\end{thm}
\begin{proof}
The proof is similar to the one of \cite[Theorem 3.2]{Cuesta_Pardo_MN_2024}, where 
we adapt the proofs of  \cite[Theorem 14.18 and Theorem 14.20]{Drabek},  and of \cite[Proposition 3.5]{Girg_Takac}, based on the well known result of Rabinowitz \cite[Theorem 1.3]{R71} to the case of a nonlinearity $g$ satisfying \eqref{crois g}.
Notice  that   the operator $G$ is compact according to Proposition \ref{pro:compG} with monotonicity only on compact sets. 

Observe that, once the  growth condition \eqref{crois g} is satisfied, we  can always redefine  an upper bound of $|g(\la ,x, s)|$, precisely as the $\max_{[0,s]} \bar{f}$, which is a non-decreasing function.
Observe also that \eqref{eq:usfF} in Theorem \ref{usfF}, implies that
$\ \lim_{s\to+\infty}\frac{sf(s)}{F(s)}=q+1>1$ for $f=\bar{f}.$
\end{proof}
\end{appendix}

\section{Acknowledgments} 

This work was done during a series of  Pardo's  visits to the Universit\'e du Littoral C\^ote d'Opale ULCO, and of   Cuesta's  visits to the Universidad Complutense de Madrid UCM, whose invitations and hospitality they thanks. The second author is supported by grants  PID2022-137074NB-I00,  MICINN,  Spain, and by UCM Spain,  Grupo 920894.

\bibliographystyle{abbrv}
\bibliography{ref}
\end{document}